\documentclass[reqno]{amsart}
\usepackage{graphicx}
\usepackage{xcolor,bbm,easybmat,bm}
\usepackage{enumerate}
\usepackage[shortlabels]{enumitem}
\usepackage{amsmath,amsfonts,amssymb,hyperref,mathrsfs,esint}
\DeclareMathOperator{\supp}{supp}
\DeclareMathOperator{\loc}{loc\ }
\DeclareMathOperator{\divg}{div}
\DeclareMathOperator{\dist}{dist}
\newtheorem{lem}[equation]{Lemma}
\newtheorem{prop}[equation]{Proposition}
\newtheorem{thm}[equation]{Theorem}
\newtheorem{cor}[equation]{Corollary}
\newtheorem{re}[equation]{Remark}
\newtheorem{defn}[equation]{Definition}

\numberwithin{equation}{section}

\newcommand{\norm}[1]{\left\Vert#1\right\Vert}
\newcommand{\abs}[1]{\left\vert#1\right\vert}
\newcommand{\br}[1]{\left(#1\right)}
\newcommand{\set}[1]{\left\{#1\right\}}
\newcommand{\Real}{\mathbb R}
\newcommand{\Rn}{\mathbb R^n}
\newcommand{\Rpl}{\mathbb R_+^{n+1}}
\newcommand{\eps}{\varepsilon}

\newcommand{\vp}{\varphi}

\newcommand{\wt}{\widetilde}
\newcommand{\bdy}{\partial}
\newcommand{\Dt}{\partial_t}
\newcommand{\Di}{\partial_i}
\newcommand{\Dj}{\partial_j}
\newcommand{\act}[1]{\langle#1\rangle_{\widetilde{W}^{-1,2},W^{1,2}}}

\newcommand{\Lrho}{\mathscr{L}_{\rho}}

\newcommand{\bd}[1]{\mathbf{#1}}

\newcommand{\1}{\mathbbm{1}}

\begin{document}


\title[$L^p$ theory for the square roots and square functions]{$L^p$ theory for the square roots and square functions of elliptic operators having a BMO anti-symmetric part}
\author{Steve Hofmann}
\author{Linhan Li}
\author{Svitlana Mayboroda}
\thanks{S.~Hofmann acknowledges support of the National Science Foundation (currently grant number DMS-1664047). S.~Mayboroda is supported in part by the NSF grants DMS 1344235, DMS 1839077 and Simons Foundation grant 563916, SM.} 
\author{Jill Pipher}

\newcommand{\Addresses}{{
  \bigskip
  \footnotesize

  Steve Hofmann, \textsc{Department of Mathematics, University of Missouri, Columbia, MO 65211, USA}\par\nopagebreak
  \textit{E-mail address}: \texttt{hofmanns@missouri.edu}

  \medskip

  Linhan Li, \textsc{School of Mathematics, University of Minnesota, Minneapolis, MN 55455, USA}\par\nopagebreak
  \textit{E-mail address}: \texttt{linhan\_li@alumni.brown.edu}

  \medskip

  Svitlana Mayboroda, \textsc{School of Mathematics, University of Minnesota, Minneapolis, MN 55455, USA}\par\nopagebreak
  \textit{E-mail address}: \texttt{svitlana@math.umn.edu}
  
  \medskip
  
   Jill Pipher, \textsc{Department of Mathematics, Brown University, Providence, RI 02906 USA}\par\nopagebreak
  \textit{E-mail address}: \texttt{jpipher@math.brown.edu}

}}

\maketitle

\begin{abstract}
We consider the operator $L=-{\rm div}(A\nabla)$, 
where $A$ is an $n\times n$ matrix of real coefficients and satisfies the ellipticity condition, with $n\ge 2$. We assume that the coefficients of the symmetric part of $A$ are in $L^\infty(\mathbb{R}^n)$, and those of the anti-symmetric part of $A$ only belong to the space $BMO(\mathbb{R}^n)$. We create a complete narrative of the $L^p$ theory for the square root of $L$ and show that it satisfies the $L^p$ estimates $\left\Vert{\sqrt{L}f}\right\Vert_{L^p}\lesssim\left\Vert{\nabla f}\right\Vert_{L^p}$ for $1<p<\infty$, and $\left\Vert{\nabla f}\right\Vert_{L^p}\lesssim\left\Vert{\sqrt{L}f}\right\Vert_{L^p}$ for $1<p<2+\epsilon$ for some $\epsilon>0$ depending on the ellipticity constant and the BMO semi-norm of the coefficients. Moreover, we prove the $L^p$ estimates for some vertical square functions associated to $e^{-tL}$. In another article of the authors, these results are used to establish the solvability of the Dirichlet problem for elliptic equation ${\rm div}(A(x)\nabla u)=0$ in the upper half-space $(x,t)\in\mathbb{R}_+^{n+1}$ with the boundary data in $L^p(\mathbb{R}^n,dx)$ for some $p\in (1,\infty)$.
\end{abstract}

\section{Introduction and main results}
\label{intro}

This paper is motivated by the study of boundary value problems for elliptic operators having a BMO anti-symmetric part. These operators arise in the study of equations with divergence-free drift, e.g. $-\Delta u+\bd c\cdot\nabla u=0$ and $\Dt u+\bd c\cdot\nabla u-\Delta u=0$, where $c$ is a divergence-free vector field in $\Rn$. Seregin, Silvestre, {\v{S}}ver{\'a}k, and Zlato{\v{s}} (\cite{seregin2012divergence}) discovered that the condition $\divg \bd c=0$ can be used to relax the regularity assumptions on $\bd c$ under which one can prove the Harnack inequality and other regularity results for solutions. It turns out that $\bd c\in BMO^{-1}$ in the elliptic case, and $\bd c\in L^{\infty}(BMO^{-1})$ in the parabolic case are the right conditions, in the sense that the interior regularity theory of De Giorgi, Nash, and Moser carries over to these operators. Generalizing to elliptic or parabolic equations in divergence form, this condition is equivalent to assuming that the matrix $A$ can be decomposed into an $L^{\infty}$ elliptic symmetric part and an unbounded anti-symmetric part in a certain function space. In the elliptic case, the anti-symmetric part should belong to the John-Nirenberg space BMO (bounded mean oscillation) and, in the parabolic case, to $L^{\infty}(BMO)$. The space BMO plays an important role in two ways. First, this space has the appropriate scaling properties which appear naturally in the iterative arguments of De Giorgi-Nash-Moser. Secondly, the BMO condition on the anti-symmetric part of the matrix allows one to define suitable weak solutions. This latter fact is essentially due to an application of the div-curl lemma appearing in the theory of compensated compactness, and the details can be found in \cite{seregin2012divergence} or \cite{li2019boundary}.

These operators have gained much attention since \cite{seregin2012divergence}. In \cite{qian2018parabolic}, the authors showed the existence of the fundamental solution of the parabolic operator $L-\partial_t$, and derived Gaussian estimate for the fundamental solution. Later, H. Dong and S. Kim \cite{dong2018fundamental} have generalized the result for fundamental solutions to second-order parabolic systems, under the assumption that weak solutions of the system satisfy a certain local boundedness estimate. The investigation into boundary value problems for elliptic operators having a BMO anti-symmetric part was launched by the work \cite{li2019boundary}. There, the second and the fourth authors of this paper studied the boundary behavior of weak solutions as well as the Dirichlet problem for elliptic operators in divergence form with BMO anti-symmetric part.

In another direction, Escauriaza and the first author of this paper proved the Kato conjecture for elliptic operators having a BMO anti-symmetric part in \cite{escauriaza2018kato}. To be precise, they showed that the domain of the square root $\sqrt{L}$ contains $W^{1,2}(\Rn)$, and that \begin{equation}\label{Kato intro}
    \norm{\sqrt{L}f}_{L^2(\Rn)}\lesssim\norm{\nabla f}_{L^2(\Rn)}
\end{equation}
holds over $\dot{W}^{1,2}(\Rn)$. Their proof does not rely on the Gaussian estimates obtained in \cite{qian2018parabolic}.
The Kato conjecture dates back to the 60's, when T. Kato conjectured (\cite{Kato1966perturbation}, \cite{kato1961fractional}) that an abstract version of \eqref{Kato intro} might hold, for ``regularly accretive operators".  The conjecture was disproved by McIntosh (\cite{mcintosh1972comparability}), who then reformulated the conjecture for divergence form elliptic operators with complex, $L^\infty$, $n\times n$ matrix. The validity of the conjecture was established when the heat kernel of the operator $L$ satisfies the ``Gaussian property", first in 2 dimensions  (\cite{hofmann2002solution2d}) and then in all dimensions \cite{hofmann2002solutionalld}.
We say $L$ satisfies the Gaussian property if the kernel $K_t(x,y)$ of the operator $e^{-tL}$ satisfies the following:
for all $t>0$, for some constants $0<\beta,\mu\le1$ and $C$,
\begin{align*}
    \abs{K_t(x,y)}&\le Ct^{-\frac{n}{2}}e^{-\frac{\beta\abs{x-y}^2}{t}},\\
    \abs{K_t(x,y)-K_t(x+h,y)}&+\abs{K_t(x,y)-K_t(x,y+h)}\\
    &\le Ct^{-\frac{n}{2}}\br{\frac{\abs{h}}{t^{1/2}+\abs{x-y}}}^\mu e^{{-\frac{\beta\abs{x-y}^2}{t}}}
\end{align*}
when $2\abs{h}\le t^{1/2}+\abs{x-y}$.
The conjecture was solved for elliptic operators in divergence form with complex, bounded coefficients in \cite{auscher2002solution}, and, as we mentioned above, for operators considered in the present paper in \cite{escauriaza2018kato}.

The $L^2$ boundeness \eqref{Kato intro} naturally leads to the question about $L^p$ boundeness, $p\neq 2$. Namely, if $L$ is such that the domain of $L^{1/2}$ agrees with $W^{1,2}(\Rn)$, how do $\norm{L^{1/2}f}_{L^p}$ and $\norm{\nabla f}_{L^p}$ compare? It turns out that the ranges of $p$ for $\norm{L^{1/2}f}_{L^p}\lesssim\norm{\nabla f}_{L^p}$ and $\norm{\nabla f}_{L^p}\lesssim\norm{L^{1/2}f}_{L^p}$ can be different. In \cite{AST_1998__249__R1_0}, it is shown that for divergence form differential operators $L=-\divg (A\nabla)$, where $A$ is a matrix with complex-valued bounded entries and satisfying a uniform ellipticity condition, if $L$ has the Gaussian property, and that \eqref{Kato intro} and its corresponding inequality for $L^*$ hold, then
\begin{equation}\label{Lp kato intro}
    \norm{L^{1/2}f}_{L^p}\le c_p\norm{\nabla f}_{L^p} \qquad\forall\,1<p<\infty,
\end{equation}
\begin{equation}\label{reverseLp Kato intro}
    \norm{\nabla f}_{L^p}\le c_p'\norm{L^{1/2}f}_{L^p} \qquad\forall\, 1<p<2+\eps,
\end{equation}
for some $\eps>0$ depends only on $L$. The proof relies on a non-standard factorization of $L^{1/2}$ which makes Calder\'on-Zygmund theory fully available. We remark that although the Gaussian property is available for elliptic operators with a BMO anti-symmetric part (\cite{qian2018parabolic}), the results in \cite{AST_1998__249__R1_0} do not apply to this setting, mainly because the decomposition used in \cite{AST_1998__249__R1_0} requires the coefficients being bounded. In \cite{auscher2007necessary}, Auscher presents the $L^p$ boundedness results without a direct appeal to kernels of the operators. The main observation in \cite{auscher2007necessary} is that the limits of the interval of exponents $p\in[1,\infty]$ for which the semigroup is $L^p$ bounded, and the limits of the interval of exponents $p\in[1,\infty]$ for which $(\sqrt{t}\nabla e^{-tL})_{t>0}$ is $L^p$ bounded, fully describe the $L^p$ behavior of the square root operator, as well as some Littlewood-Paley-Stein type functionals (we simply call them square functions in this paper). We record here that the two vertical square functions studied in \cite{auscher2007necessary} are
\[
g_L(f)(x)=\br{\int_0^\infty\abs{(L^{1/2}e^{-tL}f)(x)}^2dt}^{1/2}
\]
and
\[
G_L(f)(x)=\br{\int_0^\infty\abs{(\nabla e^{-tL}f)(x)}^2dt}^{1/2}.
\]

For $L$ being an elliptic operator having a BMO anti-symmetric part, the $L^p$ boundedness for $L^{1/2}$ and square functions was unknown. While this question is interesting by itself, we are motivated also by the study of boundary value problems for these operators. Indeed, for divergence form operators with matrix in the ``block form'', that is, $L=\divg_{x,t}(A(x)\nabla_{x,t})$ with $A=\begin{bmatrix}
\begin{BMAT}{c.c}{c.c}
B & \mathbf{0} \\
\mathbf{0}^{\intercal} & 1
\end{BMAT}
\end{bmatrix}$, inequalities \eqref{Lp kato intro} and \eqref{reverseLp Kato intro} can be thought of as a ``Rellich identity" $\norm{\Dt u}_{L^p}\approx\norm{\nabla_xu}_{L^p}$. The latter plays an important role in the solvability of Neumann and regularity problem with $L^p$ data. See e.g. \cite{jerison1981neumann}, \cite{verchota1984layer} and \cite{kenig1993neumann}. Even more generally, for operators having a full $(n+1)\times (n+1)$ elliptic coefficients matrix $A(x)$, tools related to the Kato problem have been successfully used to tackle boundary value problems. In \cite{hofmann2015square}, for instance, the $L^p$ estimates for some square functions similar to $g_L$ and $G_L$ form parts of the proof of the $L^p$ solvability for elliptic operators with real, $L^\infty$, t-independent coefficients in the upper-half space for $p$ sufficiently large. However, details regarding these $L^p$ estimates are missing, and one did not know whether these estimates are valid for elliptic operators having a BMO anti-symmetric part.

In this paper, we create a complete narrative of the $L^p$ theory for the square root operator (Section \ref{Lp semigp sqrt sec}), and derive the $L^p$ estimates for the vertical square functions (Section \ref{Lp est for square functions subsect}). Let $L=-\divg (A\nabla)$ be an operator with real coefficients defined in $\Rn$, $n\ge 2$. Assume that the symmetric part of the $n\times n$ matrix $A$ is elliptic and $L^\infty$, and the anti-symmetric part is in $BMO(\Rn)$.  Our main results are the following:
\begin{enumerate}
       \item $\norm{L^{1/2}f}_{L^p}\lesssim\norm{\nabla f}_{L^p}$ for $1<p<\infty$, and $\norm{\nabla f}_{L^p}\lesssim\norm{L^{1/2}f}_{L^p}$ for $1<p<2+\epsilon_1$. (Theorem \ref{squareroot main thm})
    \item we have the $L^p$ square functions estimates
    \begin{equation*}
        \norm{\Big(\int_0^{\infty}\abs{tL e^{-t^2L}F}^2\frac{dt}{t}\Big)^{1/2}}_{L^p(\Rn)}\le C_p\norm{\nabla F}_{L^p(\Rn)}\quad\forall\,1<p<\infty,
    \end{equation*}
    (Proposition \ref{Lp_G1Prop})
    \begin{equation*}
\norm{\Big(\int_0^{\infty}\abs{t^2\nabla L e^{-t^2L}F}^2\frac{dt}{t}\Big)^{1/2}}_{L^p(\Rn)}\le C_p\norm{\nabla F}_{L^p(\Rn)} \quad\forall\,1<p<2+\epsilon_1,
     \end{equation*}
     (Proposition \ref{Lp_G2xProp})
    \begin{equation*}
\norm{\Big(\int_0^{\infty}\abs{t^2\Dt L e^{-t^2L}F}^2\frac{dt}{t}\Big)^{1/2}}_{L^p(\Rn)}\le C_p\norm{\nabla F}_{L^p(\Rn)}\quad\forall\,1<p<\infty.
\end{equation*}
(Proposition \ref{Lp_G2t prop})
\end{enumerate}
In these results, $\epsilon_1>0$ depends only on the ellipticity constant and the BMO semi-norm of the coefficients of the operator, and on dimension (Proposition \ref{sqrttnabla_L2-LpProp}).

To deal with the BMO coefficients, we need estimates on the Hardy norm of some functions of particular form. These are presented in Section \ref{Hardy norms subsec}. We give a precise definition of the operator $L$ in Section ~\ref{operators def subsection}, starting from a sesquilinear form.
The $L^p$ estimates for the square root and square functions rely on the off-diagonal estimates for the semigroup and $(\sqrt{t}\nabla e^{-tL})_{t>0}$, which are derived in Section ~\ref{Lp semigp sqrt sec}. Another key ingredient in proving the $L^p$ estimates for the square root is the representation formula for the Riesz transform, which we carefully justify in Proposition \ref{Riesztransform prop}. To prove the $L^p$ estimates for the square functions (Proposition \ref{Lp_G1Prop}--\ref{Lp_G2t prop}), we exploit the $L^p$ estimates for the square root operator and borrow some ideas from \cite{auscher2007necessary}.

 While the paper can be viewed independently as a part of an extensive theory of functional calculus of elliptic operators and associated Hardy spaces, for us it was mainly motivated by the demands coming from the theory of boundary value problems. In \cite{HLMPdirichlet}, to continue the work \cite{li2019boundary}, we study of $L^p$ Dirichlet problem for elliptic operators having a BMO anti-symmetric part. There, we are able to prove the Dirichlet problem with $L^p(dx)$ boundary data in the upper half-space $(x,t)\in \Rpl$, $n\ge 2$, is uniquely solvable for $p$ sufficiently large, for these operators under some natural structural assumptions on the matrix, namely, $t$-independent.  In \cite{HLMPdirichlet}, we use the Gaussian estimate for the $t$-derivatives of the heat kernel \eqref{Dtl_Kt_Gaussianbd} to derive the $L^p$ estimates for some non-tangential maximal functions. The $L^p$ estimates for the square functions (Proposition \ref{Lp_G1Prop}--\ref{Lp_G2t prop}) are used to carry out a refined integration by parts argument. The result in \cite{HLMPdirichlet} extends the work of Hofmann, Kenig, Mayboroda and Pipher (\cite{hofmann2015square}), which holds for elliptic operators in divergence form with real-valued, non-symmetric, $L^\infty$ and $t$-independent coefficients.

\section{Hardy Norms}\label{Hardy norms subsec}
\begin{defn}
 We say $f\in L^1(\Rn)$ is in the real Hardy space $\mathcal{H}^1(\Rn)$ if 
 \[
 \norm{f}_{\mathcal{H}^1(\Rn)}:=\norm{\sup_{t>0}\abs{h_t*f}}_{L^1(\Rn)}<\infty,
 \]
 where $h_t(x)=\frac{1}{t^n}h\br{\frac{x}{t}}$, and $h$ is any smooth non-negative function on $\Rn$, with $\supp h\subset B_1(0)$ such that $\int_{\Rn}h(x)dx=1$.
 \end{defn}

The following estimates shall be used frequently in the rest of the paper.
\begin{prop}\label{HardyProp1}
Let $1<p<\infty$. Let $u\in\dot{W}^{1,p}(\Rn)$, $v\in \dot{W}^{1,p'}(\Rn)$. Then for any $1\le i,j\le n$, $\Dj u\Di v-\Di u\Dj v\in\mathcal{H}^1(\Rn)$ with
\begin{equation}
  \norm{\Dj u\Di v-\Di u\Dj v}_{\mathcal{H}^1(\Rn)}\lesssim \norm{\nabla u}_{L^p}\norm{\nabla v}_{L^{p'}},
\end{equation}
where the implicit constant depends only on $p$ and dimension.
\end{prop}
We refer to \cite{li2019boundary} and \cite{seregin2012divergence} for its proof.

\begin{prop}\label{HardyProp3}
Let $1<p<\infty$. Let $u\in\dot{W}^{1,p}(\Rn)$, $v\in \dot{W}^{1,p'}(\Rn)$. Then for any $1\le i\le n$, $\Di(uv)\in\mathcal{H}^1(\Rn)$ with
\begin{equation}
  \norm{\Di(uv)}_{\mathcal{H}^1(\Rn)}\lesssim \norm{u}_{L^p}\norm{\nabla v}_{L^{p'}}+\norm{\nabla u}_{L^p}\norm{v}_{L^{p'}},
\end{equation}
where the implicit constant depends only on $p$ and dimension.
\end{prop}

\begin{proof}
Let $h$ be a smooth nonnegative compactly supported mollifier with $\int_{\Rn}h(x)dx=1$, $\supp h\subset B_1(0)$. And let $h_t(x)=t^{-n}h(\frac{x}{t})$. Then we have
\begin{align*}
    \quad h_t*\Di(uv)(x)&=\int_{\Rn}h_t(x-y)\Di(uv)(y)dy=-\int_{B_t(x)}\Di h_t(x-y)u(y)v(y)dy\\
    &=\int_{B_t(x)}\frac{1}{t^{n+1}}\Di h\br{\frac{x-y}{t}}u(y)\br{v(y)-(v)_{B_t(x)}}dy\\
    &\quad+\int_{B_t(x)}\frac{1}{t^{n+1}}\Di h\br{\frac{x-y}{t}}u(y)(v)_{B_t(x)}dy
    =: I_1+I_2.
\end{align*}
For $I_1$, we have
\begin{align*}
    \abs{I_1}&\lesssim\frac{1}{t^n}\int_{B_t(x)}\abs{u(y)}\abs{\frac{v(y)-(v)_{B_t(x)}}{t}}dy\\
    &\lesssim\br{\fint_{B_t(x)}\abs{u}^{\alpha}dy}^{1/{\alpha}}\br{\fint_{B_t(x)}\abs{\frac{v(y)-(v)_{B_t(x)}}{t}}^{\alpha'}dy}^{1/{\alpha'}}\\
    &\lesssim\br{\fint_{B_t(x)}\abs{u}^{\alpha}dy}^{1/{\alpha}}\br{\fint_{B_t(x)}\abs{\nabla v}^{\beta}dy}^{1/{\beta}}
    \lesssim\br{M\abs{u}^{\alpha}}^{\frac{1}{\alpha}}(x)\br{M\abs{\nabla v}^{\beta}}^{1/{\beta}}(x),
\end{align*}
where $\alpha\in [1,p)$, $\frac{1}{\alpha}+\frac{1}{\beta}=1+\frac{1}{n}$, and $M(f)$ is the Hardy-Littlewood maximal function of $f$. For $I_2$, note that $I_2=\int_{B_t(x)}h_t(x-y)\Di u(y)(v)_{B_t(x)}dy$. So
\begin{align*}
    \abs{I_2}&\lesssim\frac{1}{t^n}\int_{B_t(x)}\abs{\Di u}\abs{(v)_{B_t(x)}}dy
    \lesssim \fint_{B_t(x)}\abs{\nabla u}dy \fint_{B_t(x)}\abs{v}dy\\
    &\lesssim M(\abs{\nabla u})(x) M(v)(x).
\end{align*}
Combining the estimates for $I_1$ and $I_2$, and using H\"older inequality, we have
\begin{multline*}
     \int_{\Rn}\sup_{t>0}\abs{h_t*\Di(uv)(x)}dx\\
    \lesssim \norm{\br{M\abs{u}^{\alpha}}^{\frac{1}{\alpha}}}_{L^p}\norm{\br{M\abs{\nabla v}^{\beta}}^{1/{\beta}}}_{L^{p'}}
    +\norm{M(\abs{\nabla u})}_{L^p}\norm{ M(v)}_{L^{p'}}\\
    \lesssim\norm{u}_{L^p}\norm{\nabla v}_{L^{p'}}+\norm{\nabla u}_{L^p}\norm{v}_{L^{p'}},
\end{multline*}
where in the last inequality we have used that $1\le \alpha<p$ and $1<\beta<p'$.
 \end{proof}

\begin{prop}\label{HardyProp2}
Let $u$, $v\in W^{1,2}(\Rn)$, and $\vp$ be a Lipschitz function in $\Rn$. Then for any $1\le i,j\le n$, $\Dj(uv)\Di\vp-\Di(uv)\Dj\vp\in\mathcal{H}^1(\Rn)$ with
\begin{equation*}
  \norm{\Dj(uv)\Di\vp-\Di(uv)\Dj\vp}_{\mathcal{H}^1(\Rn)}\lesssim \norm{u\abs{\nabla \vp}}_{L^2}\norm{\nabla v}_{L^2}+\norm{v}_{L^2}\norm{\abs{\nabla u}\abs{\nabla\vp}}_{L^2},
\end{equation*}
or
\begin{equation}
\norm{\Dj(uv)\Di\vp-\Di(uv)\Dj\vp}_{\mathcal{H}^1(\Rn)}\lesssim \norm{\nabla \vp}_{L^{\infty}(\Rn)}\Big(\norm{u}_{L^2}\norm{\nabla v}_{L^2}+\norm{v}_{L^2}\norm{\nabla u}_{L^2}\Big),
\end{equation}
where the implicit constant depends only on dimension.

\end{prop}

\begin{proof}
We can assume $\vp\in C^2(\Rn)$. Set $\vec{\Phi}=(0,\dots,0,\Dj\vp,0,\dots,0,-\Di\vp,0,\dots,0)$. Then $\divg\vec{\Phi}=0$ and
\begin{equation}\label{Hardycrosstermid}
\Di(uv)\Dj\vp-\Dj(uv)\Di\vp=\vec{\Phi}\cdot\nabla(uv)=\divg(\vec{\Phi}uv).
\end{equation}
Let $h$ be a smooth nonnegative compactly supported mollifier with $\int_{\Rn}h(x)dx=1$, $\supp h\subset B_1(0)$. And let $h_t(x)=t^{-n}h(\frac{x}{t})$. We compute
\begin{align}
h_t*\divg (\vec{\Phi}uv)(x)&=-\int_{B_t(x)}\nabla_yh_t(x-y)\cdot\vec{\Phi}(y)u(y)v(y)dy\nonumber\\
&=-\int_{B_t(x)}\nabla_yh_t(x-y)\cdot\vec{\Phi}(y)u(y)(v(y)-(v)_{B_t(x)})dy\nonumber\\
&\quad -\int_{B_t(x)}\divg_y(h_t(x-y)\vec{\Phi}(y))u(y)(v)_{B_t(x)}dy\nonumber\\
&=-\int_{B_t(x)}\nabla_yh_t(x-y)\cdot\vec{\Phi}(y)u(y)(v(y)-(v)_{B_t(x)})dy\nonumber\\
&\quad +\int_{B_t(x)}h_t(x-y)\vec{\Phi}\cdot\nabla u(y)(v)_{B_t(x)}dy\nonumber\\
&=: I_1+I_2.
\end{align}

\begin{align*}
    \abs{I_1}&\lesssim\frac{1}{t^{n+1}}\int_{B_t(x)}\abs{\nabla\vp}\abs{u}\abs{v-(v)_{B_t(x)}}\\
    &\lesssim\Big(\fint_{B_t(x)}\abs{u\nabla\vp}^{\alpha}\Big)^{1/{\alpha}}
    \Big(\fint_{B_t(x)}(\frac{\abs{v-(v)_{B_t(x)}}}{t})^{\alpha'}\Big)^{\frac{1}{\alpha'}}\\
    &\lesssim\Big(\fint_{B_t(x)}\abs{u\nabla\vp}^{\alpha}\Big)^{1/{\alpha}}
    \Big(\fint_{B_t(x)}\abs{\nabla v}^{\beta}\Big)^{1/{\beta}}
    \lesssim M^{1/{\alpha}}(\abs{u\nabla\vp}^{\alpha})(x)M^{1/{\beta}}(\abs{\nabla v}^{\beta})(x),
\end{align*}
where $1<\alpha,\beta<2$, $\frac{1}{\alpha}+\frac{1}{\alpha'}=1$, $\frac{1}{\alpha}+\frac{1}{\beta}=1+\frac{1}{n}$. And
\begin{align*}
    \abs{I_2}\lesssim\Big(\fint_{B_t(x)}\abs{\nabla\vp}\abs{\nabla u}\Big)\abs{(v)_{B_t(x)}}
    \lesssim M(\abs{\nabla\vp}\abs{\nabla u})(x)M(v)(x).
\end{align*}
So
$$
\abs{h_t*\divg (\vec{\Phi}uv)(x)}\lesssim M^{1/{\alpha}}(\abs{u\nabla\vp}^{\alpha})(x)M^{1/{\beta}}(\abs{\nabla v}^{\beta})(x)
+ M(\abs{\nabla\vp}\abs{\nabla u})(x)M(v)(x),
$$
and thus
\begin{align*}
    \int_{\Rn}\sup_{t>0}\abs{h_t*\divg(\vec{\Phi}uv)(x)}dx
    \lesssim \norm{u\abs{\nabla \vp}}_{L^2}\norm{\nabla v}_{L^2}+\norm{v}_{L^2}\norm{\abs{\nabla u}\abs{\nabla\vp}}_{L^2}.
\end{align*}
By \eqref{Hardycrosstermid} and the definition of Hardy norm, we complete the proof.
 \end{proof}

\section{Sectorial operators and resolvent estimates}
\label{operators def subsection}
We give a precise definition for the operator $L$.

Let $\widetilde{W}^{-1,2}(\Rn)$ be the space of the bounded semilinear functionals on $W^{1,2}(\Rn)$. We say that $f\in \widetilde{W}^{-1,2}$ is semilinear if
$$
\langle f, \alpha u+\beta v\rangle_{\widetilde{W}^{-1,2},W^{1,2}}=\bar{\alpha}\act{f,u}+\bar{\beta}\act{f,v},
$$
whenever $\alpha$, $\beta\in\mathbb C$ and $u$, $v\in W^{1,2}(\Rn)$.\par
Define $\mathscr{L}: W^{1,2}(\Rn)\to \widetilde{W}^{-1,2}(\Rn)$ as follows
\begin{align*}
    \act{\mathscr{L}u,v}&=\int_{\Rn}A\nabla u\cdot\nabla \bar{v}\\
    &=\int_{\Rn}A^s\nabla u\cdot\nabla \bar{v}+\int_{\Rn}A^a\nabla u\cdot\nabla \bar{v},
\end{align*}
where $A=(a_{ij}(x))$ is $n\times n$, real, $A^s=\frac{1}{2}(A+A^{\intercal})=(a^s_{ij}(x))$ is of coefficients in $L^{\infty}(\Rn)$ and elliptic, i.e. there exists $0<\lambda_0\le1$ such that for all $x\in \Rn$,
\begin{equation*}
 \lambda_0\abs{\xi}^2\le a^s_{ij}(x)\xi_i\xi_j\quad\forall\, \xi\in\Rn, \quad \norm{A^s}_{\infty}\le\lambda_0^{-1},
\end{equation*}
and the coefficients of $A^a=\frac{1}{2}(A-A^{\intercal})=(a^a_{ij}(x))$ are in BMO$(\Rn)$, with
\begin{equation*}
    \norm{a^a_{ij}}_{BMO}:=\sup_{Q\subset\Rn}\fint_{Q}\abs{a^a_{ij}-(a^a_{ij})_{Q}}dx\le\Lambda_0
\end{equation*}
for some $\Lambda_0>0$, where $Q$ is any cube in $\Rn$.
Then by Proposition \ref{HardyProp1},
\begin{equation*}
    \abs{\act{\mathscr{L}u,v}}\le C\norm{\nabla u}_{L^2}\norm{\nabla v}_{L^2},
\end{equation*}
with $C$ depending on $\lambda_0$, $\Lambda_0$ and dimension.

Now define a sesquilinear form on $L^2(\Rn)\times L^2(\Rn)$: for any $u$, $v\in W^{1,2}(\Rn)$, let
\begin{equation*}
    t[u,v]=\act{\mathscr{L}u,v}.
\end{equation*}
The numerical range $\Theta(t)$ of $t$ is defined as
$$
\Theta(t):=\{t[u,u]: u\in D(t) \text{ with } \norm{u}_{L^2}=1\}.
$$
\begin{prop}\label{nr_t_Prop}
$t$ is a densely defined, closed, sectorial sesqulinear form in $L^2$, and there exists $0<\theta_0<\frac{\pi}{2}$ such that for any $\xi\in\Theta(t)$, $\abs{\arg\xi}\le\theta_0$.
\end{prop}
\begin{proof}
The domain $D(t)$ of $t$ is $W^{1,2}(\Rn)$, which is dense in $L^2$. So $t$ is densely defined. To see that it is closed, let $u_n\in D(t)$, $u_n\to u$ in $L^2$ and $t[u_n-u_m,u_n-u_m]\to 0$. We want to show that $u\in D(t)$ and $t[u_n-u,u_n-u]\to 0$. Since $t[u_n-u_m,u_n-u_m]\to 0$,
$$
\lambda_0\abs{\nabla(u_n-u_m)}\le\Re\int_{\Rn}A\nabla(u_n-u_m)\cdot\nabla\overline{(u_n-u_m)}\to 0.
$$
So $\{u_n\}$ is a Cauchy sequence in $W^{1,2}(\Rn)$, which implies that $u\in W^{1,2}=D(t)$ and
$$
\abs{t[u_n-u,u_n-u]}\le\Lambda_0\norm{\abs{\nabla(u_n-u)}}_{L^2}\to 0.
$$
Now we show that $t$ is sectorial, i.e., its numerical range $\Theta(t)$ is a subset of a sector of the form
\begin{equation*}
    \abs{\arg(\xi-\gamma)}\le\theta, \quad\text{for some } 0\le\theta<\frac{\pi}{2} \text{ and } \gamma\in\Real.
\end{equation*}
For $u\in D(t)$ with $\norm{u}_{L^2}=1$, write $u=u_1+iu_2$. Then
\begin{equation*}
    \Re \act{\mathscr{L}u, u}=\int_{\Rn}a^s_{ij}(\Dj u_1\Di u_1+\Dj u_2\Di u_2)\ge\lambda_0\int_{\Rn}\abs{\nabla u}^2,
\end{equation*}
\begin{equation*}
    \Im\act{\mathscr{L}u, u}=\int_{\Rn}a^s_{ij}(\Dj u_2\Di u_1-\Dj u_1\Di u_2)+\int_{\Rn}a^a_{ij}(\Dj u_2\Di u_1-\Dj u_1\Di u_2).
\end{equation*}
By Proposition \ref{HardyProp1}, $\abs{\Im\act{\mathscr{L}u,u}}\le C\int_{\Rn}\abs{\nabla u}^2$, with $C$ depending on $\lambda_0$, $\Lambda_0$ and $n$. This implies that
$$
\frac{\abs{\Im\act{\mathscr{L}u,u}}}{\Re \act{\mathscr{L}u, u}}\le C, \text{ with } C=C(\lambda_0, \Lambda_0, n).
$$
Therefore, there exists $0<\theta_0=\theta_0(\lambda_0, \Lambda_0, n)<\frac{\pi}{2}$ such that for any $\xi\in\Theta(t)$, $\abs{\xi}\le\theta_0$.
 \end{proof}

Then by \cite{kato1976perturbation} Chapter VI Theorem 2.1 and its proof, we obtain
\begin{lem}
There is a unique m-accretive, sectorial operator $L:D(L)\subset L^2(\Rn)\to L^2(\Rn)$ such that
\begin{enumerate}
    \item $D(L)\subset D(t)=W^{1,2}(\Rn)$, and $D(L)$ is dense in $D(t)$ with respect to the $W^{1,2}$ norm.
    \item $(Lu,v)=t[u,v]$ for all $u\in D(L)$, $v\in D(t)$. Here $(\cdot,\cdot)$ is the inner product on complex $L^2(\Rn)$.
    \item If $u\in D(t)$, $w\in L^2(\Rn)$, and $t[u,v]=(w,v)$ for any $v\in D(L)$, then $u\in D(L)$ and $Lu=w$.
\end{enumerate}
By m-accretive we mean that $(L+\lambda I)^{-1}$ is a bounded operator on $L^2$ for any $\Re\,\lambda>0$, and $\norm{(L+\lambda I)^{-1}}_{L^2\to L^2}\le (\Re\,\lambda)^{-1}$ .
\end{lem}
This lemma implies that there is a unique $L: D(L)\subset L^2(\Rn)\to L^2(\Rn)$ with its domain dense in $W^{1,2}$ corresponding to $\mathscr{L}: W^{1,2}\to \widetilde{W}^{-1,2}$, and $(Lu,v)=\act{\mathscr{L}u,v}$ for any $u\in D(L)$, $v\in W^{1,2}$.

We denote by $\Theta(L)$ the numerical range of $L$:
$$
\Theta(L)=\{(Lu,u): u\in D(L) \text{ with } \norm{u}_{L^2}=1.\}.
$$
Let $T=-L$. Denote the resolvent set of $T$ by $\rho(T)$. Note that since $L$ is m-accretive, $\{\lambda\in\mathbb C:\Re\,\lambda<0\}\subset\rho(L)$. So $\lambda\in\rho(T)$ whenever $\Re\,\lambda>0$. Let $\Sigma_0=\mathbb{C}\setminus\overline{\Theta(T)}$, and denote by $\hat{\Sigma}_0$ the component of $\Sigma_0$ that contains $\Real^+$.
\begin{lem}\label{LemresolventL2}
$\hat{\Sigma}_0\subset\rho(T)$. And for any $\lambda\in\hat{\Sigma}_0$,
\begin{equation*}
    \norm{(\lambda I-T)^{-1}}_{L^2\to L^2}=\norm{(\lambda I+L)^{-1}}_{L^2\to L^2}\le\frac{1}{\dist(\lambda;\overline{\Theta(T)})}.
\end{equation*}
\end{lem}
\begin{proof}
For any fixed $\lambda\in\Sigma_0$, for any $u\in D(T)$ with $\norm{u}_{L^2}=1$, we have
\[
    0<\dist(\lambda,\overline{\Theta(T)})\le\abs{\lambda-(Tu,u)}=\abs{((\lambda I-T)u,u)}\le\norm{(\lambda I-T)u}_{L^2}.
\]
Hence, if $\lambda\in\rho(T)$, then
\begin{equation}\label{resolventL2bound1}
    \norm{(\lambda I-T)^{-1}}_{L^2\to L^2}\le \frac{1}{\dist(\lambda;\overline{\Theta(T)})}.
\end{equation}
Now we show $\hat{\Sigma}_0\subset\rho(T)$. Consider $\rho(T)\cap\hat{\Sigma}_0$. It is nonempty since $\Real^+\subset\rho(T)\cap\hat{\Sigma}_0$. The fact that $\rho(T)$ is open implies that $\rho(T)\cap\hat{\Sigma}_0$ is open in $\hat{\Sigma}_0$. But it is also closed in $\hat{\Sigma}_0$ since $\lambda_n\in\rho(T)\cap\hat{\Sigma}_0$ and $\lambda_n\to\lambda\in\hat{\Sigma}_0$ imply for $n$ large enough, $\dist(\lambda_n,\overline{\Theta(T)})>\frac{1}{2}\dist(\lambda,\overline{\Theta(T)})$, and consequently for $n$ large enough
$\abs{\lambda_n-\lambda}<\dist(\lambda_n,\overline{\Theta(T)})$. Write $\lambda I-T=(\lambda_n I-T)(I+(\lambda-\lambda_n)(\lambda_n I-T)^{-1})$. From \eqref{resolventL2bound1},
\[
    \norm{(\lambda-\lambda_n)(\lambda_n I-T)^{-1}}\le\abs{\lambda-\lambda_n}\norm{(\lambda_n I-T)^{-1}}
    \le \frac{\abs{\lambda-\lambda_n}}{\dist(\lambda_n;\overline{\Theta(T)})}<\frac{1}{2},
\]
which implies that $(I+(\lambda-\lambda_n)(\lambda_n I-T)^{-1})^{-1}$ is bounded in $L^2$, and consequently so is $(\lambda I-T)^{-1}$, i.e. $\lambda\in\rho(T)$. This implies that $\rho(T)\cap\hat{\Sigma}_0$ is closed in $\hat{\Sigma}_0$. By the connectedness of $\hat{\Sigma}_0$, $\rho(T)\cap\hat{\Sigma}_0=\hat{\Sigma}_0$, or $\hat{\Sigma}_0\subset\rho(T)$.
 \end{proof}
Fix a $\theta_1\in(\theta_0,\frac{\pi}{2})$. Let
$\Gamma_{\pi-\theta_1}=\{\lambda\in\mathbb{C}:\lambda\neq0, \abs{\arg{\lambda}}\le\pi-\theta_1\}$. Then $\Gamma_{\pi-\theta_1}\subset\rho(T)$ and there exists a $c_0=c_0(\theta_0,\theta_1)\ge1$ such that for any $\lambda\in\Gamma_{\pi-\theta_1}$,
\begin{equation}\label{lambda_dist_theta1}
    \dist(\lambda;\overline{\Theta(T)})\ge\frac{\abs{\lambda}}{c_0}.
\end{equation}
\begin{cor}\label{L_resolv_cor}
There exists a $C=C(\theta_0,\theta_1,\lambda_0)>0$ such that
for any $\lambda\in\Gamma_{\pi-\theta_1}$,
$$
\norm{(\lambda I+L)^{-1}}_{L^2\to L^2}+\frac{1}{\abs{\lambda}^{1/2}}\norm{\nabla(\lambda I+L)^{-1}}_{L^2\to L^2}\le \frac{C}{\abs{\lambda}}.
$$
\end{cor}
\begin{proof}
It follows immediately from Lemma \ref{LemresolventL2} and \eqref{lambda_dist_theta1} that
\begin{equation*}
\norm{(\lambda I+L)^{-1}}_{L^2\to L^2}\le\frac{c_0}{\abs{\lambda}}.
\end{equation*}
Since $\Gamma_{\pi-\theta_1}\subset\rho(T)$, for any $\lambda\in\Gamma_{\pi-\theta_1}$, the range of $\lambda I+L$ is $L^2(\Rn)$. Also, for any $u\in D(L)$, $((\lambda I+L)u,u)=\act{(\mathscr{L}+\lambda)u,u}$, which gives that
\begin{equation*}
    \Re ((\lambda I+L)u,u)\le\norm{(\lambda I+L)u}_{L^2}\norm{u}_{L^2}.
\end{equation*}
On the other hand,
\begin{equation*}
    \Re ((\lambda I+L)u,u)=\Re(\int_{\Rn}A\nabla u\cdot\nabla\bar{u}+\lambda\int_{\Rn}\abs{u}^2)
    \ge\lambda_0\norm{\nabla u}^2_{L^2}+\Re\lambda\norm{u}_{L^2}^2.
\end{equation*}
Therefore,
\begin{equation*}
    \lambda_0\int_{\Rn}\abs{\nabla u}^2\le\abs{\lambda}\norm{u}_{L^2}^2+\norm{(\lambda I+L)u}_{L^2}\norm{u}_{L^2}.
\end{equation*}
Since
$$
\norm{u}_{L^2}=\norm{(\lambda I+L)^{-1}(\lambda I+L)u}_{L^2}\le\frac{c_0}{\abs{\lambda}}\norm{(\lambda I+L)u}_{L^2},
$$
one obtains $\norm{\nabla u}_{L^2}\le C(c_0,\lambda_0)\abs{\lambda}^{-1/2}\norm{(\lambda I+L)u}_{L^2}$, and consequently
\begin{equation*}
    \norm{\nabla(\lambda I+L)^{-1}}\le C(c_0,\lambda_0)\abs{\lambda}^{-1/2}.
\end{equation*}
 \end{proof}

\section{$L^p$ theory for the semigroup and square roots}
\label{Lp semigp sqrt sec}

The resolvent estimates we derived in Corollary \ref{L_resolv_cor} imply that there exists an analytic contraction semigroup $e^{-tL}$ on $L^2$ generated by $-L$. The semigroup can be expressed as the contour integral (see e.g. \cite{pazy1983semigroups} Chapter 1, Theorem 7.7)
\begin{equation*}
    e^{-tL}=\frac{1}{2\pi i}\int_{\Gamma}e^{t\lambda}(L+\lambda I)^{-1}d\lambda,
\end{equation*}
where the path $\Gamma$ consists of two half-rays $\Gamma_{\pm}=\set{\lambda=re^{\pm i(\pi-\theta_1)},r\ge R}$ and of the arc $\Gamma_0=\set{\lambda=Re^{i\theta},\abs{\theta}\le\pi-\theta_1}$, for any fixed $R>0$, $\theta_1\in (\theta_0,\frac{\pi}{2})$, where $\theta_0$ is as in Proposition \ref{nr_t_Prop}. It follows that $e^{-tL}$ is holomorphic in $t$ in the open sector $\abs{\arg t}<\frac{\pi}{2}-\theta_0$.

Let $K_t(x,y)$ be the kernel of $e^{-tL}$.
The results by Qian and Xi \cite{qian2018parabolic} show that the operator $L$ satisfies the ``Gaussian Property", that is,
$K_t(x,y)$ satisfies the following bounds: there are some constants $C=C(\lambda_0,\Lambda_0,n)$, $\beta=\beta(\lambda_0,\Lambda_0,n)\in (0,1)$ and $\mu_0=\mu_0(\lambda_0,\Lambda_0,n)\in (0,1)$, such that for all $t>0$, $0<\mu<\mu_0$
\begin{align}
    \abs{K_t(x,y)}&\le Ct^{-\frac{n}{2}}e^{-\frac{\beta\abs{x-y}^2}{t}},\label{K_t_bd}\\
    \abs{K_t(x,y)-K_t(x+h,y)}&+\abs{K_t(x,y)-K_t(x,y+h)}\nonumber\\
    &\quad\le Ct^{-\frac{n}{2}}\br{\frac{\abs{h}}{t^{1/2}+\abs{x-y}}}^{\mu}e^{{-\frac{\beta\abs{x-y}^2}{t}}}\label{K_t_Holderbd}
\end{align}
when $2\abs{h}\le t^{1/2}+\abs{x-y}$.

We remark that the Gaussian estimates for the $t$-derivatives of the kernel $K_t(x,t)$ can be derived from the Gaussian estimates (\cite{ouhabaz2009analysis} Theorem 6.17):
\begin{prop}
For any $l\in\mathbb{N}$, $\partial_t^l K_t(x,y)$ satisfies the following estimates: there are some constants $C=C(\lambda_0,\Lambda_0,n)$, $\beta=\beta(\lambda_0,\Lambda_0,n)\in (0,1)$, such that for all $t>0$,
\begin{align}
    \abs{\partial_t^l K_t(x,y)}&\le C_l t^{-\frac{n}{2}-l}e^{-\frac{\beta\abs{x-y}^2}{t}}\label{Dtl_Kt_Gaussianbd}
\end{align}

\end{prop}

The $L^2$ (actually, $L^p$ for all $1\le p<\infty$) estimate for $\partial_t^l e^{-tL}$ follows from \eqref{Dtl_Kt_Gaussianbd} immediately: for all $t>0$,
\begin{equation}\label{DtetL_L2bd}
    \norm{\partial_t^l e^{-tL}}_{L^2\to L^2}\le c_l t^{-l}\qquad\text{for all } l\in\mathbb{N}.
\end{equation}

Moreover, we have for all $t>0$,
\begin{align}\label{nablaetL_L2bd}
        \norm{\nabla e^{-tL}}_{L^2\to L^2}\le Ct^{-\frac{1}{2}}.
\end{align}

\eqref{nablaetL_L2bd} can be derived from \eqref{DtetL_L2bd} as follows:
\begin{align*}
    \norm{\nabla e^{-tL}f}_{L^2}^2&\le\frac{1}{\lambda_0}\Re \int_{\Rn}A\nabla e^{-tL}\cdot \overline{\nabla e^{-tL}f}\\
    &\le\frac{1}{\lambda_0}\norm{\Dt e^{-tL}f}_{L^2}\norm{e^{-tL}f}_{L^2}\le C\,t^{-1}\norm{f}_{L^2}^2.
\end{align*}

\subsection{$L^2$ off-diagonal estimates for the semigroup}

\begin{defn}
Let $\mathcal{T}=(T_t)_{t>0}$ be a family of operators. We say that $\mathcal{T}$ satisfies $L^2$ off-diagonal estimates if for some constants $C\ge0$ and $\alpha>0$ for all closed sets $E$ and $F$, all $h\in L^2$ with support in $E$ and all $t>0$ we have
\begin{equation}\label{off-diag_Defn ineq}
  \norm{T_th}_{L^2(F)}\le Ce^{-\frac{\alpha d(E,F)^2}{t}}\norm{h}_{L^2}.
\end{equation}
Here and subsequntly, $d(E,F)$ is the semi-distance induced on sets by the Euclidean distance.

If $\mathcal{T}=(T_z)_{z\in\Sigma_{\mu}}$ is a family defined on a complex sector $\Sigma_{\mu}$ with $0\le\mu<\frac{\pi}{2}$, then we adopt the same definition and replace $t$ by $\abs{z}$ in the right hand side of \eqref{off-diag_Defn ineq}. In this case, the constants $C$ and $\alpha$ may depend on the angle $\mu$.
\end{defn}

\begin{prop}\label{L2off-diag Prop}
There exists $\omega_0=\omega_0(n,\lambda_0,\Lambda_0)\in(0,\frac{\pi}{2})$, such that for all $\mu\in(0,\frac{\pi}{2}-\omega_0)$, the families $(e^{-zL})_{z\in\Sigma_{\mu}}$, $(zLe^{-zL})_{z\in\Sigma_{\mu}}$ and $(\sqrt{z}\nabla e^{-zL})_{z\in\Sigma_{\mu}}$ satisfy $L^2$ off-diagonal estimates.
\end{prop}
\begin{proof}
We begin with the case of real times $t>0$. Let $\vp$ be a bounded Lipschitz function with Lipschitz constant 1 and $\rho>0$. Define
$\mathscr{L}_{\rho}=e^{\rho\vp}\mathscr{L}e^{-\rho\vp}$ as follows:  for any $u,v\in W^{1,2}(\Rn)$
\begin{equation*}
    \act{\Lrho u,v}=\act{\mathscr{L}(e^{\rho\vp}u),e^{-\rho\vp}v}.
\end{equation*}
Note that since $\norm{\nabla\rho}_{L^{\infty}}$ is bounded, Proposition \ref{HardyProp2} implies that $\mathscr{L}_{\rho}: W^{1,2}\to \widetilde{W}^{-1,2}$ is bounded.
Define $\Lrho'=\Lrho+c_0\rho^2$, with $c_0$ to be determined. Using Proposition \ref{HardyProp2}, we estimate
\begin{multline*}
    \Re\act{\Lrho'u,u}\\
    \ge\Re\act{\mathscr{L}u,u}-C_{n,\Lambda_0}\norm{\rho\nabla\vp}_{L^{\infty}}\norm{\nabla u}_{L^2}\norm{u}_{L^2}\\
    +\br{c_0\rho^2-\frac{1}{\lambda_0}\norm{\rho\nabla\vp}^2_{L^{\infty}}}\norm{u}_{L^2}^2\\
    \ge\frac{\lambda_0}{2}\norm{\nabla u}_{L^2}^2+\br{c_0\rho^2-C_{n,\lambda_0,\Lambda_0}\rho^2}\norm{u}_{L^2}^2.
\end{multline*}

Here we have used $\norm{\nabla\vp}_{L^\infty}\le1$. Now by choosing $c_0=c_0(n,\lambda_0,\Lambda_0)$ sufficiently large, we have
$$
\Re\act{\Lrho'u,u}\ge\frac{\lambda_0}{2}\norm{\nabla u}_{L^2}^2+\frac{c_0}{2}\rho^2\norm{u}_{L^2}^2.
$$
For $\Im\act{\Lrho'u,u}$, we have
\begin{align*}
    \abs{\Im\act{\Lrho'u,u}}&\le\abs{\Im\act{\mathscr{L}u,u}}+C_{\lambda_0}\norm{\nabla u}_{L^2}^2+C_{\lambda_0}\rho^2\norm{u}_{L^2}^2\\
    &\le C_{n,\lambda_0,\Lambda_0}\norm{\nabla u}_{L^2}^2+C_{\lambda_0}\rho^2\norm{u}_{L^2}^2.
\end{align*}
Therefore, there exists $C_0=C_0(n,\lambda_0,\Lambda_0)>0$, such that
\[
\frac{\abs{\Im\act{\Lrho'u,u}}}{\Re\act{\Lrho'u,u}}\le C_0,
\]
which implies that there exists $0<\omega_0=\omega_0(n,\lambda_0, \Lambda_0)<\frac{\pi}{2}$ such that
\begin{equation}\label{omega_0}
   \forall\, \xi\in\Theta(\Lrho'),  \quad\abs{\xi}\le\omega_0.
\end{equation}
As we argue in Section \ref{operators def subsection}, we can find unique sectorial operators $L_{\rho}'$ and $L_{\rho}$ corresponds to $\Lrho'$ and $\Lrho$, respectively. Then we can prove estimates \eqref{DtetL_L2bd} and \eqref{nablaetL_L2bd} for the operator $L_{\rho}'$. That is, there is $C=C(n,\lambda_0,\Lambda_0)$, such that
\begin{equation*}
    \norm{e^{-tL_{\rho}'}(f)}_{L^2}+\norm{t\Dt e^{-tL_{\rho}'}(f)}_{L^2}+\norm{\sqrt{t}\nabla e^{-tL_{\rho}'}(f)}_{L^2}\le C\norm{f}_{L^2}
\end{equation*}
for all $t>0$. Then by $L_{\rho}'=L_{\rho}+c_0\rho^2I$ and $\Dt e^{-tL_{\rho}'}=-L_{\rho}'e^{-tL_{\rho}'}$, a direct computation shows
\begin{equation*}
    \norm{e^{-tL_{\rho}}(f)}_{L^2}+\norm{t(L_{\rho}+c_0\rho^2I)e^{-tL_{\rho}}(f)}_{L^2}+\norm{\sqrt{t}\nabla e^{-tL_{\rho}}(f)}_{L^2}\le Ce^{c_0\rho^2t}\norm{f}_{L^2}.
\end{equation*}
This implies
\begin{equation}\label{etLrho_L2}
   \norm{e^{-tL_{\rho}}(f)}_{L^2}+\norm{t\Dt e^{-tL_{\rho}}(f)}_{L^2}+\norm{\sqrt{t}\nabla e^{-tL_{\rho}}(f)}_{L^2}\le Ce^{2c_0\rho^2t}\norm{f}_{L^2} \quad\forall\, t>0.
\end{equation}

Let $E$ and $F$ be two closed sets and $f\in L^2$, with compact support contained in $E$. For any $\eps>0$, choose $\vp(x)=\vp_{\eps}(x)=\frac{d(x,E)}{1+\eps d(x,E)}$. With this choice of $\vp_{\eps}$, the operator $\mathscr{L}_{\rho}$ has bounded form.
Observe that there exists $R_0>1$ sufficiently large, such that for any $R\ge R_0$,
\begin{equation*}
    d(E,F)\le d(E,F\cap B_R(0))\le 2d(E,F).
\end{equation*}
Also, for any fixed $R\ge R_0$, there exists $\eps_0=\eps_0(n,R,E)>0$, such that for any $0<\eps\le\eps_0$, $\eps d(x,E)\le\frac{1}{2}$ for all $x\in F\cap B_R(0)$, and thus
\begin{equation}\label{vp_eps ge}
    \vp_{\eps}(x)=\frac{d(x,E)}{1+\eps d(x,E)}\ge \frac{2}{3}d(x,E), \quad\forall\, x\in F\cap B_R(0).
\end{equation}
With these observations and \eqref{etLrho_L2}, the argument of the $L^2$ off-diagonal estimates for $(e^{-tL})_{t>0}$, $(t\Dt e^{-tL})_{t>0}$ and $(\sqrt{t}\nabla e^{-tL})_{t>0}$ follows from \cite{auscher2007necessary} Proposition 3.1.

To extend to complex times, consider the operator $e^{i\alpha}\mathscr{L}$, which has coefficients $e^{i\alpha}A(x)$. Note that if $\xi$ is in the numerical range of $e^{i\alpha}\mathscr{L}$, then $\abs{\arg\xi}\le\theta_0+\abs{\alpha}$, where the angle $\theta_0$ is as in \eqref{nr_t_Prop}. Therefore, in light of \eqref{omega_0}, the argument above applies to $e^{i\alpha}\mathscr{L}$ as long as $\abs{\alpha}<\frac{\pi}{2}-\omega_0$. Observe that when $z=te^{i\alpha}$, $e^{-zL}=e^{-t(e^{i\alpha}L)}$. From this the desired estimates follow.
 \end{proof}

\begin{re}
The same argument applies to the adjoint operator $L^*$. Therefore, for all $\mu\in(0,\frac{\pi}{2}-\omega_0)$, the families $(e^{-zL^*})_{z\in\Sigma_{\mu}}$, $(zL^*e^{-zL^*})_{z\in\Sigma_{\mu}}$, and $(\sqrt{z}\nabla e^{-zL^*})_{z\in\Sigma_{\mu}}$ satisfy $L^2$ off-diagonal estimates.
\end{re}

Using the $L^2$ off-diagonal estimates, we can define the action of the semigroup on $L^{\infty}$ and on Lipschitz function in the $L_{\loc}^2$ sense.

\begin{lem}\label{etL_Linfty_lemDefn}
Let $f\in L^{\infty}(\Rn)$. Then for any $x_0\in\Rn$,  $\underset{R\to\infty}{\lim}e^{-tL}(f\1_{B_R(x_0)})$ exists in $L_{\loc}^2(\Rn)$ and the limit does not depend on $x_0$. We define the limit to be $e^{-tL}f$.
\end{lem}

\begin{proof}
We first fix any $x_0\in\Rn$. Fix any $R_0>1$, and let $R_2>R_1>8R_0$. Then there exists $l\in\mathbb N$ such that $2^l R_1<R_2\le 2^{l+1}R_1$. We write
\begin{align*}
    &\quad\abs{e^{-tL}(f\1_{B_{R_2}(x_0)})-e^{-tL}(f\1_{B_{R_1}(x_0)})}=\abs{e^{-tL}\br{f\1_{B_{R_2}(x_0)}-f\1_{B_{R_1}(x_0)}}}\\
    &\le \abs{e^{-tL}f\br{\1_{B_{R_2}(x_0)}-\1_{B_{2^lR_1}(x_0)}}}
    +\sum_{k=1}^l\abs{e^{-tL}f\br{\1_{B_{2^kR_1}(x_0)}-\1_{B_{2^{k-1}R_1}(x_0)}}}.
\end{align*}
Observe that $d(B_{R_0}(x_0),B_{2^kR_1}(x_0)\setminus B_{2^{k-1}R_1}(x_0))\gtrsim 2^{k-1}R_1$ for $k=1,2,\dots, l$. Then by $L^2$ off-diagonal estimates for $\br{e^{-tL}}_{t>0}$, we have
\begin{align*}
    \Big|\Big|e^{-tL}&(f\1_{B_{R_2}(x_0)})-e^{-tL}(f\1_{B_{R_1}(x_0)})\Big|\Big|_{L^2(B_{R_0}(x_0))}\\
    &\lesssim e^{-\frac{c(2^lR_1)^2}{t}}\norm{f}_{L^2(B_{R_2}(x_0)\setminus B_{2^lR_1}(x_0))}
    +\sum_{k=1}^l e^{-\frac{c(2^{k-1}R_1)^2}{t}}\norm{f}_{L^2(B_{2^kR_1}(x_0))}\\
    &\lesssim\sum_{k=1}^{l+1}e^{-\frac{c(2^{k-1}R_1)^2}{t}}(2^kR_1)^{n/2}\norm{f}_{L^{\infty}}\lesssim t^{n/2}R_1^{-n/2}\norm{f}_{L^{\infty}},
\end{align*}
where the implicit constant depends only on $\lambda_0$, $\Lambda_0$ and $n$. This shows that $e^{-tL}(f\1_{B_{R}(x_0)})$ is a Cauchy sequence in $L_{\loc}^2(\Rn)$, when $f\in L^{\infty}$.

We now show the limit is independent of the choice of $x_0$. Let $x_1\in\Rn$ be a different point than $x_0$. Then for $R$ sufficiently large, the symmetric difference $B_{R}(x_0)\Delta B_{R}(x_1)$ is contained in $B_{2R}(0)\setminus B_{\frac{R}{2}}(0)$. So by the $L^2$ off-diagonal estimates for $e^{-tL}$, and that $d(B_{R}(x_0)\Delta B_{R}(x_1),B_{\frac{R}{4}}(0))\gtrsim R$, we have
\begin{align*}
    \norm{e^{-tL}f\1_{B_R(x_0)}-e^{-tL}f\1_{B_R(x_1)}}_{L^2(B_{\frac{R}{4}}(0))}&\lesssim e^{-\frac{cR^2}{t}}\norm{f}_{L^2(B_R(x_0)\Delta B_R(x_1))}\\
    &\lesssim R^{\frac{n}{2}}e^{-\frac{cR^2}{t}}\norm{f}_{L^{\infty}}.
\end{align*}
This implies that $\lim_{R\to\infty}e^{-tL}f\1_{B_R(x_0)}=\lim_{R\to\infty}e^{-tL}f\1_{B_R(x_1)}$ in the $L_{\loc}^2$ sense.
 \end{proof}

\begin{re}
The limit $\lim_{R\to\infty}e^{-tL}(f\1_{B_R(x_0)})$ actually exists in $W_{\loc}^{1,2}(\Rn)$. One can show this by using the $L^2$ off-diagonal estimates for $\br{\sqrt{t}\nabla e^{-tL}f}_{t>0}$ instead of the $L^2$ off-diagonal estimates for $\br{e^{-tL}}_{t>0}$ in the argument above.
\end{re}

Similarly, we can define $e^{-tL}f$ for $f$ Lipschitz.
\begin{lem}\label{etL_Lipshcitz_lemDefn}
Let $f$ be a Lipschitz function. Then for any $x_0\in\Rn$, the limit $\lim_{R\to\infty}e^{-tL}(f\1_{B_R(x_0)})$ exists in $L_{\loc}^2(\Rn)$ and not depend on $x_0$. We define the limit to be $e^{-tL}f$.
\end{lem}

\begin{proof}
Fix any $x_0\in\Rn$, any $R_0>1$, and let $R_2>R_1>8R_0$. Then there exists $l\in\mathbb N$ such that $2^l R_1<R_2\le 2^{l+1}R_1$. We write
\begin{multline*}
    \abs{e^{-tL}(f\1_{B_{R_2}(x_0)})-e^{-tL}(f\1_{B_{R_1}(x_0)})}\\
    \le\abs{e^{-tL}\br{(f-f(x_0))(\1_{B_{R_2}(x_0)}-\1_{B_{R_1}(x_0)})}}\\+\abs{e^{-tL}\br{f(x_0)(\1_{B_{R_2}(x_0)}-\1_{B_{R_1}(x_0)})}}
    =: I_1+I_2
\end{multline*}
Since $f(x_0)$ is a bounded constant function, the proof of Lemma \ref{etL_Linfty_lemDefn} applies to $I_2$, and we have
\begin{equation}\label{etL_Lip_I2}
  \norm{I_2}_{L^2(B_{R_0}(x_0))}\lesssim t^{n/2}R_1^{-n/2}\abs{f(x_0)}.
\end{equation}
For $I_1$, we have
\begin{multline*}
    \norm{I_1}_{L^2(B_{R_0}(x_0))}
    \lesssim e^{-c\frac{(2^lR_1)^2}{t}}\norm{(f-f(x_0))\1_{B_{R_2}(x_0)\setminus B_{2^lR_1}(x_0)}}_{L^2}\\
    +\sum_{k=1}^le^{-c\frac{(2^{k-1}R_1)^2}{t}}\norm{(f-f(x_0))\1_{B_{2^kR_1}(x_0)\setminus B_{2^{k-1}R_1}(x_0)}}_{L^2}\\
    \lesssim\sum_{k=1}^{l+1}e^{-c\frac{(2^{k-1}R_1)^2}{t}}(2^kR_1)^{1+\frac{n}{2}}\norm{\nabla f}_{L^{\infty}(\Rn)}\\
    \lesssim\sum_{k=1}^{\infty}\br{\frac{(2^kR_1)^2}{t}}^{-\frac{n+1}{2}}(2^kR_1)^{1+\frac{n}{2}}\norm{\nabla f}_{L^{\infty}}\\
    \lesssim t^{\frac{n+1}{2}}R_1^{-\frac{n}{2}}\norm{\nabla f}_{L^{\infty}}.
\end{multline*}

This and \eqref{etL_Lip_I2} show that $e^{-tL}(f\1_{B_{R}(x_0)})$ is a Cauchy sequence in $L_{\loc}^2(\Rn)$, when $f$ is Lipschitz. By a similar argument as in the proof of Lemma \ref{etL_Linfty_lemDefn}, one can show the limit is independent of $x_0$.
 \end{proof}

\begin{re}
By applying the $L^2$ off-diagonal estimates for $\br{\sqrt{t}\nabla e^{-tL}f}_{t>0}$ instead of the $L^2$ off-diagonal estimates for $\br{e^{-tL}}_{t>0}$ in the argument above, one can show that as $R\to\infty$, $e^{-tL}(f\1_{B_{R}(x_0)})$ converges in $W_{\loc}^{1,2}(\Rn)$, when $f$ is Lipschitz.
\end{re}

\begin{prop}
The conservation property $e^{-tL}1=1$, for any $t>0$, holds in the sense of $L_{\loc}^2$.
\end{prop}

\begin{proof}
Let $\Phi$ be an $L^2(\Rn)$ function with compact support, and suppose the support of $\Phi$ is contained in a cube $Q$ with $l(Q)=r_0$. We first show $e^{-tL}\Phi\in L^1$.

Let $Q_0=2Q$, and decompose $\Rn$ into a union of nonoverlapping cubes with same size. That is, $\Rn=\bigsqcup_{k=0}^{\infty}Q_k$, with $\abs{Q_k}=\abs{Q_0}$. Let $x_Q$ denote the center of $Q$ and $Q_0$, and let $x_k$ denote the center of $Q_k$, for $k=1,2,\dots$. Define
\[
\mathcal{F}_l=\set{Q_k: 2lr_0\le\abs{x_k-x_Q}<2(l+1)r_0}, \quad\text{for }l=1,2,\dots.
\]
By the $L^2$ boundedness of the semigroup, we have
\begin{equation}\label{etL_L1_Q_0}
    \int_{Q_0}\abs{e^{-tL}\Phi}\le\abs{Q_0}^{1/2}\norm{e^{-tL}\Phi}_{L^2}\lesssim\abs{Q_0}^{1/2}\norm{\Phi}_{L^2}.
\end{equation}

For any $Q_k\in\mathcal{F}_l$, we use the $L^2$ off-diagonal estimates for $(e^{-tL})_{t>0}$ to obtain
\[
\int_{Q_k}\abs{e^{-tL}\Phi}\le\abs{Q_k}^{1/2}\norm{e^{-tL}\Phi}_{L^2(Q_k)}\lesssim \abs{Q_0}^{1/2}e^{-\frac{c(lr_0)^2}{t}}\norm{\Phi}_{L^2}.
\]
Since $\abs{\mathcal{F}_l}\approx l^n$, we have
\begin{align*}
   \sum_{Q_k\in\mathcal{F}_l}\int_{Q_k}\abs{e^{-tL}\Phi}\lesssim l^n \abs{Q_0}^{1/2}e^{-\frac{c(lr_0)^2}{t}}\norm{\Phi}_{L^2}
   \lesssim l^{-2}r_0^{-\frac{n}{2}-2}t^{\frac{n+2}{2}}\norm{\Phi}_{L^2}.
\end{align*}
Summing in $l$ yields
\begin{equation}\label{etL_L1_Q_0cpl}
    \int_{\Rn\setminus Q_0}\abs{e^{-tL}\Phi}\lesssim\sum_{l=1}^{\infty}l^{-2}r_0^{-\frac{n}{2}-2}t^{\frac{n+2}{2}}\norm{\Phi}_{L^2}
    \lesssim r_0^{-\frac{n}{2}-2}t^{\frac{n+2}{2}}\norm{\Phi}_{L^2}.
\end{equation}
Then $e^{-tL}\Phi\in L^1$ follows from \eqref{etL_L1_Q_0} and \eqref{etL_L1_Q_0cpl}. The argument also applies to $e^{-tL^*}$ and so $e^{-tL^*}\Phi\in L^1$. Therefore, by Lemma \ref{etL_Linfty_lemDefn}, we have
\begin{align}\label{etL1_id}
    \int_{\Rn}e^{-tL}1\overline{\Phi}=\lim_{R\to\infty}\int_{\Rn}e^{-tL}(\1_{B_R})\overline{\Phi}
    =\lim_{R\to\infty}\int_{\Rn}\1_{B_R}\overline{e^{-tL^*}\Phi}=\int_{\Rn}\overline{e^{-tL^*}\Phi}.
\end{align}
Here and subsequently, $B_R$ is the ball centered at the origin with radius $R$. We shall first show $\int_{\Rn}e^{-tL}1\overline{\Phi}$ does not depend on $t>0$ by proving
\begin{equation}\label{dint_etL1=0}
    \frac{d}{dt}\int_{\Rn}e^{-tL}1\overline{\Phi}=0,
\end{equation}
and then show
\begin{equation}\label{etL1=1 L2loc}
    \int_{\Rn}e^{-tL}1\overline{\Phi}=\int_{\Rn}\overline{\Phi}.
\end{equation}
This implies that $e^{-tL}1=1$ in the sense of $L^2_{\loc}$.

Observe that we can define $\Dt e^{-tL}f=\lim_{R\to\infty}\Dt e^{-tL}(f\1_{B_R(x_0)})$ in $L_{\loc}^2$, for $f\in L^{\infty}$, for any $x_0\in\Rn$. This is because the argument in the proof of Lemma \ref{etL_Linfty_lemDefn} applies to $\Dt e^{-tL}$ if one uses the $L^2$ off-diagonal estimates for $\br{tLe^{-tL}}_{t>0}$ instead of that for $\br{e^{-tL}}_{t>0}$. In particular, we have $\Dt e^{-tL}1=\lim_{R\to\infty}\Dt e^{-tL}\1_{B_R}$. Also, using the $L^2$ off-diagonal estimates for $\br{tL^*e^{-tL^*}}_{t>0}$, one can show $\Dt e^{-tL^*}\Phi\in L^1$. Therefore,
\[
\int_{\Rn}\Dt e^{-tL}1\overline{\Phi}=\lim_{R\to\infty}\int\Dt e^{-tL}\1_{B_R}\overline{\Phi}
=\lim_{R\to\infty}\int \1_{B_R}\overline{\Dt e^{-tL^*}\Phi}=\int\overline{\Dt e^{-tL^*}\Phi}.
\]

Let $\eta\in C_0^{\infty}(\Rn)$ with $\eta=1$ in $B_1$, and $\supp\eta\subset B_2$. Let $\eta_R(x)=\eta\br{\frac{x}{R}}$ for $R>0$.
Then
\begin{equation}\label{dtetL_id}
  \frac{d}{dt}\int_{\Rn}e^{-tL}1\overline{\Phi}=\int_{\Rn}\Dt e^{-tL}1\overline{\Phi}=\int_{\Rn}\eta_R\overline{\Dt e^{-tL^*}\Phi}+\int_{\Rn}(1-\eta_R)\overline{\Dt e^{-tL^*}\Phi}.
\end{equation}
Since $\Dt e^{-tL^*}\Phi\in L^1$, the last term goes to 0 as $R\to\infty$. We write
\begin{align*}
   \int_{\Rn}\eta_R\overline{\Dt e^{-tL^*}\Phi}&=-\int_{\Rn}\eta_R\overline{L^* e^{-tL^*}\Phi}
    =\int_{\Rn}A\nabla\eta_R\cdot\overline{\nabla e^{-tL^*}\Phi}\\
    &=\int_{\Rn}A^s\nabla\eta_R\cdot\overline{\nabla e^{-tL^*}\Phi}
    +\int_{\Rn}A^a\nabla\eta_R\cdot\overline{\nabla\br{\zeta_R e^{-tL^*}\Phi}}.
\end{align*}
Here, $\zeta_R(x)=\zeta\br{\frac{x}{R}}$, where $\zeta\in C_0^{\infty}(\Rn)$ with $\zeta=1$ in $B_2\setminus B_1$ and $\supp\zeta\subset B_{5/2}\setminus B_{1/2}$. Note that $\zeta_R=1$ in the support of $\nabla\eta_R$.

Choose $R$ to be sufficiently large so that $Q\subset B_{\frac{R}{8}}$.
We estimate
\begin{multline*}
    \abs{\int_{\Rn}A^s\nabla\eta_R\cdot\overline{\nabla e^{-tL^*}\Phi}}\le\frac{1}{\lambda_0}\norm{\nabla\eta_R}_{L^2}\norm{\nabla e^{-tL^*}\Phi}_{L^2(B_{2R}\setminus B_R)}\\
    \lesssim R^{\frac{n}{2}-1}t^{-\frac{1}{2}}e^{-\frac{cR^2}{t}}\norm{\Phi}_{L^2},
\end{multline*}
where the last inequality follows from the $L^2$ off-diagonal estimates for $\br{\sqrt{t}\nabla e^{-tL^*}}_{t>0}$. By Proposition \ref{HardyProp1}, we have
\[
\abs{\int_{\Rn}A^a\nabla\eta_R\cdot\overline{\nabla\br{\zeta_R e^{-tL^*}\Phi}}}\le C\Lambda_0\norm{\nabla\eta_R}_{L^2}\norm{\nabla\br{\zeta_R e^{-tL^*}\Phi}}_{L^2}.
\]
Using the support property of $\zeta_R$ and $\nabla\zeta_R$, we have
\begin{align*}
   \norm{\nabla\br{\zeta_R e^{-tL^*}\Phi}}_{L^2}&\lesssim\norm{\nabla(e^{-tL^*}\Phi)}_{L^2(B_{\frac{5R}{2}}\setminus B_{\frac{R}{2}})}+R^{-1}\norm{e^{-tL^*}}_{L^2(B_{\frac{5R}{2}}\setminus B_{\frac{R}{2}})}\\
   &\lesssim t^{-\frac{1}{2}}e^{-\frac{cR^2}{t}}\norm{\Phi}_{L^2}+R^{-1}e^{-\frac{cR^2}{t}}\norm{\Phi}_{L^2},
\end{align*}
where the last inequality follows from the $L^2$ off-diagonal estimates for $\br{\sqrt{t}\nabla e^{-tL^*}}_{t>0}$ and $\br{e^{-tL^*}}_{t>0}$. Combining these estimates, we obtain $\int_{\Rn}\eta_R\overline{\Dt e^{-tL^*}\Phi}\to 0$ as $R\to \infty$. So from \eqref{dtetL_id}, the desired result \eqref{dint_etL1=0} follows.

To prove \eqref{etL1=1 L2loc}, we fix $R>0$ sufficiently large so that $2Q\subset B_R$. Write
\[
\int_{\Rn}e^{-tL}1\overline{\Phi}=\int_{\Rn}\eta_R\overline{e^{-tL^*}\Phi}+\int_{\Rn}(1-\eta_R)\overline{e^{-tL^*}\Phi}.
\]
Since $e^{-tL^*}$ is strongly continuous in $L^2$ at $t=0$,
\[
\lim_{t\to 0} \int_{\Rn}\eta_R\overline{e^{-tL^*}\Phi}= \int_{\Rn}\eta_R\overline{\Phi}=\int\overline{\Phi}.
\]

We have $\abs{\int_{\Rn}(1-\eta_R)\overline{e^{-tL^*}\Phi}}\le \int_{\Rn\setminus Q_0}\abs{e^{-tL^*}\Phi}$, where $Q_0$ is constructed in the beginning. Since we can also obtain \eqref{etL_L1_Q_0cpl} for $e^{-tL^*}$, we then have
\[\abs{\int_{\Rn}(1-\eta_R)\overline{e^{-tL^*}\Phi}}\lesssim r_0^{-\frac{n}{2}-2}t^{\frac{n+2}{2}}\norm{\Phi}_{L^2},\]
which goes to $0$ as $t\to0$. Therefore, since we have shown that $\int_{\Rn}e^{-tL}1\overline{\Phi}$ is independent of $t$, we obtain \eqref{etL1=1 L2loc}.
 \end{proof}

\subsection{$L^p$ theory for the semigroup}

We now study the uniform boundedness of the semigroup $(e^{-tL})_{t>0}$ and of the family $(\sqrt{t}\nabla e^{-tL})_{t>0}$ on $L^p$ spaces. We begin with a few definitions.

Let $\mathcal{T}=(T_t)_{t>0}$ be a family of uniformly bounded operators on $L^2$.
\begin{defn}
 We say that $\mathcal{T}$ is $L^p-L^q$ bounded for some $p,q\in [1,\infty]$ with $p\le q$ if for some constant $C$, for all $t>0$ and all $h\in L^p\cap L^2$
\[
\norm{T_th}_{L^q}\le Ct^{-\frac{\gamma_{pq}}{2}}\norm{h}_{L^p},
\]
where $\gamma_{pq}=\abs{\frac{n}{q}-\frac{n}{p}}$. We shall use $\gamma_p$ to denote $\gamma_{p2}=\abs{\frac{n}{2}-\frac{n}{p}}$.
\end{defn}

\begin{defn}
We say that $\mathcal{T}$ satisfies $L^p-L^q$ off-diagonal estimates for some $p,q\in [1,\infty]$ with $p\le q$ if for some constants $C$, $c>0$, for all closed sets $E$ and $F$, all $h\in L^p\cap L^2$ with support in $E$ and all $t>0$ we have
\[
\norm{T_th}_{L^q(F)}\le Ct^{-\frac{\gamma_{pq}}{2}}e^{-\frac{cd(E,F)^2}{t}}\norm{h}_{L^p}.
\]

\end{defn}

Note that the uniform $L^p$ boundedness of $(e^{-tL})_{t>0}$ and of $(t\Dt e^{-tL})_{t>0}$ follows from the kernel estimates \eqref{Dtl_Kt_Gaussianbd}.

\begin{lem}
Let $p\ge1$. There is some constant $C=C(n,\lambda_0,\Lambda_0,p)$ such that for all $t>0$ and all $f\in L^p$,
\begin{align}
    \norm{e^{-tL}f}_{L^p}&\le C\norm{f}_{L^p},\label{etL_Lp}\\
    \norm{t\Dt e^{-tL}f}_{L^p}&\le C\norm{f}_{L^p}.\label{tDtetL-Lp}
\end{align}
\end{lem}
\begin{proof}
By \eqref{Dtl_Kt_Gaussianbd}, we have
\begin{align}
    \int_{\Rn}\abs{K_t(x,y)}dx\le C &\qquad \int_{\Rn}\abs{K_t(x,y)}dy\le C,\label{K_t_L1bound}\\
    \int_{\Rn}\abs{\Dt K_t(x,y)}dx\le Ct^{-1} &\qquad \int_{\Rn}\abs{\Dt K_t(x,y)}dy\le Ct^{-1}.\label{DtK_t_L1bound}
\end{align}
Then
\begin{multline*}
    \abs{e^{-tL}f(x)}=\abs{\int_{\Rn}K_t(x,y)f(y)dy}\\
    \le\br{\int_{\Rn}\abs{K_t(x,y)}dy}^{\frac{1}{p'}}\br{\int_{\Rn}\abs{K_t(x,y)}\abs{f(y)}^pdy}^{\frac{1}{p}}.
\end{multline*}
So \eqref{etL_Lp} follows from \eqref{K_t_L1bound}. And \eqref{tDtetL-Lp} follows from \eqref{DtK_t_L1bound} by the same argument.
 \end{proof}

\begin{prop}\leavevmode\label{etL_Lp-L2Prop}
\begin{enumerate}
    \item $(e^{-tL})_{t>0}$ is $L^p-L^2$ bounded for any $1\le p<2$.
    \item $(e^{-tL})_{t>0}$ satisfies the $L^p-L^2$ off-diagonal estimates for any $1<p<2$.
    \item $(e^{-tL})_{t>0}$ is $L^2-L^p$ bounded for $2<p\le\infty$, and satisfies the $L^2-L^p$ off-diagonal estimates for $2<p<\infty$.
\end{enumerate}
\end{prop}

(1) is a consequence of the $L^p$ boundedness of the semigroup, and the Gagliardo-Nirenberg inequality. Details can be found in \cite{auscher2007necessary} Proposition 4.2. Item (2) follows from interpolating by the Riesz-Thorin theorem the $L^p-L^2$ boundeness with the $L^2$ off-diagonal estimates of the semigroup. And since (1) and (2) also holds for $e^{-tL^*}=(e^{-tL})^*$, (3) follows from duality.

The next results for $(t\Dt e^{-tL})_{t>0}$ are in the same spirit of Propostion \ref{etL_Lp-L2Prop}. We give the proof in full detail.
\begin{prop}\leavevmode\label{tDtetL_Lp-L2Prop}
\begin{enumerate}
    \item $(t \Dt e^{-tL})_{t>0}$ is $L^p-L^2$ bounded for any $1\le p<2$.
    \item $(t\Dt e^{-tL})_{t>0}$ satisfies the $L^p-L^2$ off-diagonal estimates for any $1<p<2$.
    \item $(t \Dt e^{-tL})_{t>0}$ is $L^2-L^p$ bounded for $2<p\le\infty$, and satisfies the $L^2-L^p$ off-diagonal estimates for $2<p<\infty$.
\end{enumerate}
\end{prop}

\begin{proof}
(1). Let $p\in [1,2)$. Since $\Dt e^{-tL}f\in W^{1,2}(\Rn)$, 
we can apply the Gagliardo-Nirenberg inequality and get,
\begin{equation}\label{DtetL_GN_ineq}
    \norm{\Dt e^{-tL}f}_{L^2}^2\le C\norm{\nabla \Dt e^{-tL}f}_{L^2}^{2\alpha}\norm{\Dt e^{-tL}f}_{L^p}^{2\beta}
\end{equation}
for all $t>0$ and $f\in L^2\cap L^p$, where
\(
\alpha+\beta=1 \text{ and } (1+\gamma_p)\alpha=\gamma_p.
\)
Using $\Dt^2 e^{-tL}f\in L^2$ and  Lemma \ref{Evanslem}, we have
\begin{multline*}
    \Re \int_{\Rn}A\nabla\Dt e^{-tL}f\cdot\overline{\nabla\Dt e^{-tL}f}
    =\Re\br{L\Dt e^{-tL}f, \Dt e^{-tL}f}\\
    =-\Re\act{\Dt^2 e^{-tL}f,e^{-tL}f}
    =-\frac{1}{2}\frac{d}{dt}\norm{\Dt e^{-tL}f}_{L^2}^2.
\end{multline*}
By ellipticity,
\begin{equation}\label{nablaDt_bd_Dt}
    \norm{\nabla \Dt e^{-tL}f}_{L^2}^2\le-\frac{1}{2\lambda_0}\frac{d}{dt}\norm{\Dt e^{-tL}f}_{L^2}^2.
\end{equation}
Assume $f\in L^2\cap L^p$ with $\norm{f}_{L^p}=1$. Let $\vp(t)=\norm{\Dt e^{-tL}f}$. By the $L^p$ boundedness of $(\Dt e^{-tL})_{t>0}$, $\norm{\Dt e^{-tL}f}_{L^p}\le C_pt^{-1}$. Then by \eqref{DtetL_GN_ineq} and \eqref{nablaDt_bd_Dt}, one obtains
\(
t^{\frac{2\beta}{\alpha}}\vp(t)^{\frac{1}{\alpha}}\le -C\vp'(t)\).
Integrating in $t$,
\[
 \int_t^{2t}t^{\frac{2\beta}{\alpha}}dt\le-C\int_t^{2t}\frac{\vp'(t)}{\vp(t)^{1/\alpha}}dt, 
\]
and thus $\vp(t)\le Ct^{-\frac{2-\alpha}{1-\alpha}}$.
Here we assumed that $\vp(t)\neq0$. Otherwise, considering $\vp(t)+\eps$ and then letting $\eps\to0$ would give the same result. Thus
$\norm{t\Dt e^{-tL}f}_{L^2}^2\le Ct^{-\frac{\alpha}{1-\alpha}}=Ct^{-\gamma_p}$, which proves (1).

(2). As in the proof of Proposition \ref{etL_Lp-L2Prop}, it follows from the $L^2$ off-diagonal estimates of $(t\Dt e^{-tL})_{t>0}$, the $L^1-L^2$ boundedness of $(t\Dt e^{-tL})_{t>0}$, and the Riesz-Thorin interpolation theorem.

(3). Since $e^{-tL}f\in D(L)$ for any $f\in L^2$, $t\Dt e^{-tL}f=-tLe^{-tL}f=-te^{-tL}Lf$. So we have
\[
(t\Dt e^{-tL})^*=-tL^*e^{-tL^*}=t\Dt e^{-tL^*}.
\]
Since (2) and (3) also hold for $t\Dt e^{-tL^*}$, (3) follows from duality and a limiting argument.
 \end{proof}

For $(\sqrt{t}\nabla e^{-tL})_{t>0}$, when $p<2$, we immediately obtain the $L^p-L^2$ boundedness and the $L^p-L^2$ off-diagonal estimates. We include the short proofs here for the sake of completeness.
\begin{prop}\leavevmode
\begin{enumerate}
    \item $(\sqrt{t}\nabla e^{-tL})_{t>0}$ is $L^p-L^2$ bounded for any $1\le p<2$.
    \item $(\sqrt{t}\nabla e^{-tL})_{t>0}$ satisfies $L^p-L^2$ off-diagonal estimates for any $1<p<2$.
\end{enumerate}
\end{prop}
\begin{proof}
Let $p\in[1,2)$ and $f\in L^2\cap L^p$. Write $\sqrt{t}\nabla e^{-tL}f=\sqrt{t}\nabla e^{-\frac{tL}{2}}e^{-\frac{tL}{2}}f$. Then by the $L^2$ boundedness of $(\sqrt{t}\nabla e^{-tL})_{t>0}$ and the $L^p-L^2$ boundedness of $(e^{-tL})_{t>0}$, one has
\[
\norm{\sqrt{t}\nabla e^{-tL}f}_{L^2}\le C \norm{e^{-\frac{tL}{2}}f}_{L^2}\le C_pt^{-\frac{\gamma_p}{2}}\norm{f}_{L^p},
\]
which proves (1).

(2) follows from the $L^2$ off-diagonal estimates of $(\sqrt{t}\nabla e^{-tL})_{t>0}$, the $L^1-L^2$ boundedness of $(\sqrt{t}\nabla e^{-tL})_{t>0}$, and the Riesz-Thorin interpolation theorem.
 \end{proof}

When $p>2$, a duality argument would not give us the desired results as in Proposition \ref{etL_Lp-L2Prop} (3) and Proposition \ref{tDtetL_Lp-L2Prop} (3). However, we are able to derive a reverse H\"older type inequality for $\nabla e^{-tL}f$, and then use the $L^2-L^p$ boundedness of $(t\Dt e^{-tL})_{t>0}$ to obtain the $L^2-L^p$ boundedness of $(\sqrt{t}\nabla e^{-tL})_{t>0}$. Note that this approach is entirely different from the existing proof (see e.g. \cite{auscher2007necessary} chapter 4), as the latter relies on the boundedness of the coefficients and does not work for BMO coefficients.

\begin{prop}\leavevmode\label{sqrttnabla_L2-LpProp}
\begin{enumerate}
    \item $(\sqrt{t}\nabla e^{-tL})_{t>0}$ is $L^2-L^p$ bounded for any $2\le p\le2+\epsilon_1$ for some $\epsilon_1=\epsilon_1(\lambda_0,\Lambda_0,n)>0$.
    \item $(\sqrt{t}\nabla e^{-tL})_{t>0}$ satisfies the $L^2-L^p$ off-diagonal estimates for any $2\le p<2+\epsilon_1$, where $\epsilon_1$ is as in (1).
\end{enumerate}
\end{prop}

\begin{proof}
Let $f\in\mathscr{S}(\Rn)$, and let $u(x,t)=e^{-tL}f(x)$. Then $u$ satisfies the equation $\Dt u+Lu=0$ in $L^2$. That is, for any $w\in W^{1,2}(\Rn)$,
\[
\int_{\Rn}A(x)\nabla u(x,t)\cdot\overline{\nabla w(x)}dx=-\int_{\Rn}\Dt u(x,t)\overline{w(x)}dx, \quad\forall\, t>0.
\]
Fix $t>0$ and fix a cube $Q\subset\Rn$ with $l(Q)=\rho_0$, where $\rho_0$ is to be determined. Let $x_0\in 3Q$ and let
$0<\rho<\min\set{\frac{1}{2}\dist(x_0,\bdy(3Q)),\rho_0}$. Let $Q_s(x)$ denote the cube centered at $x$ with side length $s$. Choose $\vp\in C_0^2(\Rn)$, with $0\le\vp\le1$, $\vp=1$ in $Q_{\rho}(x_0)$, $\supp\vp\subset Q_{\frac{3}{2}\rho}(x_0)$, and $\abs{\nabla\vp}\lesssim\frac{1}{\rho}$.

Let $w(x)=\br{u(x,t)-c}\vp^2(x)$, where $c=\fint_{Q_{2\rho}(x_0)}u(x,t)dx$. Then 
\[
\int_{\Rn}A(x)\nabla u(x,t)\cdot\nabla\br{(u(x,t)-c)\vp^2(x)}dx=-\int_{\Rn}\Dt u(x,t)\br{u(x,t)-c}\vp^2(x)dx.
\]
We have
\begin{align*}
    \int_{\Rn}A^s\nabla u\cdot\nabla\br{(u-c)\vp^2}dx&=\int_{\Rn}A^s\nabla u\cdot\nabla u\vp^2dx+2\int_{\Rn}A^s\nabla u\cdot\nabla\vp\br{(u-c)\vp}dx\\
    &\ge \frac{\lambda_0}{2}\int_{\Rn}\abs{\nabla u}^2\vp^2- C\int_{\Rn}\abs{u-c}^2\abs{\nabla\vp}^2\\
    &\ge \frac{\lambda_0}{2}\int_{Q_{\rho}(x_0)}\abs{\nabla u}^2-C\rho^{-2}\int_{Q_{\frac{3}{2}\rho}(x_0)}\abs{u-c}^2.
\end{align*}
To deal with $\int_{\Rn}A^a\nabla u\cdot\nabla\br{(u-c)\vp^2}dx$, we introduce another bump function $\eta\in C_0^2(\Rn)$ with $0\le\eta\le1$, $\eta=1$ on $Q_{\frac{3}{2}\rho}(x_0)$, $\supp\eta\subset Q_{2\rho}(x_0)$, and $\abs{\nabla\eta}\lesssim\frac{1}{\rho}$. Then we have
\begin{align*}
    \int_{\Rn}A^a\nabla u\cdot\nabla\br{(u-c)\vp^2}dx&=\int_{\Rn}A^a\nabla u\cdot\nabla(\vp^2)(u-c)dx\\
    &=\frac{1}{2}\int_{\Rn}A^a\nabla\br{(u-c)^2\eta^2}\cdot\nabla(\vp^2).
\end{align*}
By Proposition \ref{HardyProp2}, 
\begin{multline*}
    \abs{\int_{\Rn}A^a\nabla u\cdot\nabla\br{(u-c)\vp^2}dx}\le C\norm{\nabla\vp}_{L^{\infty}}\norm{(u-c)\eta}_{L^2}\norm{\nabla \br{(u-c)\eta}}_{L^2}\\
    \le\frac{\lambda_0}{4}\int_{\Rn}\abs{\nabla u}^2\eta^2dx+C\norm{\nabla\vp}_{L^{\infty}}^2\int_{\Rn}\abs{u-c}^2\eta^2dx\\+C\norm{\nabla\vp}_{L^{\infty}}\int_{\Rn}\abs{u-c}^2\abs{\nabla\eta}^2dx\\
    \le\frac{\lambda_0}{4}\int_{Q_{2\rho}(x_0)}\abs{\nabla u}^2dx+C\rho^{-2}\int_{Q_{2\rho}(x_0)}\abs{u-c}^2dx.
\end{multline*}
For $\int_{\Rn}\Dt u(x,t)\br{u(x,t)-c}\vp^2(x)dx$, we use Cauchy-Schwarz inequality to get
\begin{multline*}
    \abs{\int_{\Rn}\Dt u(x,t)\br{u(x,t)-c}\vp^2(x)dx}\\
    \le\frac{1}{2}\rho^2\int_{\Rn}\abs{\Dt u}^2\vp^2dx+\frac{1}{2}\rho^{-2}\int_{\Rn}\abs{u-c}^2\vp^2dx\\
    \le\frac{1}{2}\rho_0^2\int_{Q_{2\rho}(x_0)}\abs{\Dt u}^2dx+\frac{1}{2}\rho^{-2}\int_{Q_{2\rho}(x_0)}\abs{u-c}^2dx.
\end{multline*}
Combining these estimates, we obtain
\begin{multline*}
    \frac{\lambda_0}{2}\int_{Q_{\rho}(x_0)}\abs{\nabla u}^2dx\\
    \le\frac{C}{\rho^2}\int_{Q_{2\rho}(x_0)}\abs{u-c}^2dx+\frac{\lambda_0}{4}\int_{Q_{2\rho}(x_0)}\abs{\nabla u}^2dx+\frac{\rho_0^2}{2}\int_{Q_{2\rho}(x_0)}\abs{\Dt u}^2dx.
\end{multline*}
The Sobolev-Poincar\'e inequality gives
\begin{multline*}
    \fint_{Q_{\rho}(x_0)}\abs{\nabla u}^2dx\\
    \le C\br{\fint_{Q_{2\rho}(x_0)}\abs{\nabla u}^{\frac{2n}{n+2}}}^{\frac{n+2}{n}}+\frac{1}{2}\fint_{Q_{2\rho}(x_0)}\abs{\nabla u}^2dx+C\rho_0^2\fint_{Q_{2\rho}(x_0)}\abs{\Dt u}^2dx.
\end{multline*}
Then by Lemma \ref{GiaProp1.1}, there is some $\epsilon_1=\epsilon_1(\lambda_0,\Lambda_0,n)>0$, such that for all $p\in[2,2+\epsilon_1]$,
\[
\br{\fint_{Q}\abs{\nabla u}^pdx}^{1/p}\le C\Big\{\br{\fint_{2Q}\abs{\nabla u}^2}^{1/2}+\br{\fint_{2Q}\abs{\rho_0\Dt u}^p}^{1/p}\Big\},
\]
that is,
\begin{equation}\label{RH_Q}
    \int_Q\abs{\nabla u}^pdx\le C\rho_0^{-p\gamma_p}\br{\int_{2Q}\abs{\nabla u}^2}^{p/2}+\int_{2Q}\abs{\rho_0\Dt u}^pdx.
\end{equation}

Decompose $\Rn$ into a union of disjoint cubes $\Rn=\sqcup_{j=1}^{\infty}Q_j$ with each $Q_j$ having side length $\rho_0$. For each $Q_j$, applying \eqref{RH_Q} and then summing in $j$, one has
\[
\norm{\nabla u}_{L^p(\Rn)}\le C\rho_0^{-\gamma_p}\norm{\nabla u}_{L^2(\Rn)}+C\norm{\rho_0\Dt u}_{L^p(\Rn)}, \quad\forall\, p\in[2,2+\epsilon_1].
\]
Choosing $\rho_0=\sqrt{t}$ gives
\[
\norm{\sqrt{t}\nabla u}_{L^p}\le Ct^{-\frac{\gamma_p}{2}}\norm{\sqrt{t}\nabla u}_{L^2}+C\norm{t\Dt u}_{L^p}, \quad\forall\, p\in[2,2+\epsilon_1].
\]
Then by the $L^2$ boundedness of $(\sqrt{t}\nabla e^{-tL})_{t>0}$ and the $L^2-L^p$ boundedness of $(t\Dt e^{-tL})_{t>0}$, we obtain
\[
\norm{\sqrt{t}\nabla e^{-tL}f}_{L^p}\le Ct^{-\frac{\gamma_p}{2}}\norm{f}_{L^2}, \qquad\forall\,p\in[2,2+\epsilon_1],\quad f\in\mathscr{S}(\Rn).
\]
Thus (1) follows from a standard limiting argument.

(2) can be proved using the $L^2$ off-diagonal estimates and the $L^2-L^{2+\epsilon_1}$ boundedness of $(\sqrt{t}\nabla e^{-tL})_{t>0}$, and the Riesz-Thorin interpolation theorem. 
\end{proof}

\subsection{$L^p$ Theory for the square root}\label{Lp Kato subsec}
Since $L$ is an m-accretive operator, there is a unique m-accretive square root $L^{1/2}$ such that
\begin{equation}\label{sqrt_defid}
    L^{1/2}L^{1/2}=L \qquad\text{in}\quad D(L).
\end{equation}
Also, $L^{1/2}$ is m-sectorial with the numerical range contained in the sector $\abs{\arg\xi}\le\frac{\pi}{4}$. And $D(L)$ is a {\it core} of $L^{1/2}$, i.e. $\set{(u,L^{1/2}u):u\in D(L)}$ is dense in the graph $\set{(u, L^{1/2}u):u\in D(L^{1/2})}$ (see \cite{kato1976perturbation} p.281 for a proof for these facts).

Our goal in this section is to prove the $L^p$ bounds for the square root.

Many formulas can be used to compute $L^{1/2}$. The one we are going to use is
\begin{equation}\label{squareroot_formula}
    L^{1/2}f=\pi^{-1/2}\int_0^{\infty}e^{-tL}Lf\frac{dt}{\sqrt{t}}.
\end{equation}
Observe that the integral converges in $L^2$ when $f\in D(L)$. Since for $f\in D(L)$, $Lf\in L^2$, then by the $L^2$ boundedness of the semigroup, $\int_0^1e^{-tL}Lf\frac{dt}{\sqrt{t}}$ converges. And the $L^2$ bound of $(t\Dt e^{-tL})_{t>0}$ implies that $\int_1^{\infty}e^{-tL}Lf\frac{dt}{\sqrt{t}}$ converges in $L^2$.

The determination of the domain of the square root of $L$ has become known as the Kato square root problem. It has been shown by Auscher, Hofmann, Lacey, McIntosh, and Tchamitchian \cite{auscher2002solution} that for a uniformly complex elliptic operator $L =-\divg(A\nabla)$ with bounded measurable coefficients, one has in all dimensions
\begin{equation}\label{sqaurerootL2_equiv}
    \norm{L^{1/2}f}_{L^2}\approx\norm{\nabla f}_{L^2},
\end{equation}
and the domain of $L^{1/2}$ is $W^{1,2}$, which was known as the Kato's conjecture. Recently, Escauriaza and Hofmann (\cite{escauriaza2018kato}) extended the result to the same kind of operators that we are interested in, that is, operators with a BMO anti-symmetric part. Note that although they only showed one side of \eqref{sqaurerootL2_equiv}, that is
\begin{equation}\label{Kato}
    \norm{L^{1/2}f}_{L^2}\lesssim\norm{\nabla f}_{L^2},
\end{equation}
the other direction follows from a duality argument.
In fact, note that the same argument applies to $L^*$ so one has
\begin{equation}\label{Kato*}
   \norm{(L^*)^{1/2}f}_{L^2}\lesssim\norm{\nabla f}_{L^2}.
\end{equation}
It turns out that \eqref{Kato} and \eqref{Kato*} are enough:

\begin{lem}
If \eqref{Kato} holds for all $f\in D(L)$, and \eqref{Kato*} holds for all $f\in D(L^*)$. Then
\begin{equation}\label{Kato_reverse}
    \norm{\nabla f}_{L^2}\lesssim\norm{L^{1/2}f}_{L^2}, \qquad\forall\, f\in W^{1,2}.
\end{equation}
And the domain of $L^{1/2}$ is $W^{1,2}(\Rn)$.
\end{lem}
\begin{proof}

We first show that $W^{1,2}\subset D(L^{1/2})$. Since $D(L)$ is dense in $W^{1,2}$, for any $u\in W^{1,2}$, there are $\set{u_k}\subset D(L)$ such that $u_k\to u$ in $W^{1,2}$. Then by \eqref{Kato}, $\norm{L^{1/2}(u_k-u_j)}_{L^2}\lesssim\norm{\nabla (u_k-u_j)}_{L^2}$. This shows that $\set{L^{1/2}u_k}$ is Cauchy in $L^2$. Suppose $L^{1/2}u_k\to v\in L^2$. Since $L^{1/2}$ is closed, we have $L^{1/2}u=v$, and $u\in D(L^{1/2})$.

Now we show \eqref{Kato_reverse} holds. Let $f\in \mathscr{S}(\Rn)$. Let $\bd g\in \mathscr{S}(\Rn)$ with $\norm{\bd g}_{L^2}\le1$. For any $\delta>0$, define
\(h_{\delta}:= (L^*+\delta I)^{-1}(-\divg\bd g)\in D(L^*)\). That is,
\begin{equation}\label{hdelta}
    \delta\int_{\Rn}h_{\delta}\,\overline{w}\,dx+\int_{\Rn}A^*\nabla h_{\delta}\cdot\overline{\nabla w}\, dx=-\int_{\Rn}\divg \bd{g}\,\overline{w}\,dx\qquad\forall\, w\in W^{1,2}(\Rn).
\end{equation}
Letting $w=h_{\delta}$ and taking real parts of \eqref{hdelta}, then ellipticity and Young's inequality give
\begin{equation}\label{hdelta L2gradient}
   \delta\int_{\Rn}\abs{h_{\delta}}^2dx+ \frac{\lambda_0}{2}\int_{\Rn}\abs{\nabla h_{\delta}}^2dx\le C\int_{\Rn}\abs{\bd g}^2dx\le C.
\end{equation}

By writing
\begin{align*}
    \br{\nabla f,\bd g}&=-\br{f,\divg\bd g}=\br{f,(L^*+\delta I)h_{\delta}}=(f,(L^*)^{1/2}(L^*)^{1/2}h_{\delta})+\delta(f,h_{\delta})\\
    &=(L^{1/2}f,(L^*)^{1/2}h_{\delta})+\delta(f,h_{\delta}),
\end{align*}
we get that
\begin{align*}
   \abs{(\nabla f,\bd g)}&\le \norm{L^{1/2}f}_{L^2}\norm{(L^*)^{1/2}h_{\delta}}_{L^2}+\delta \norm{f}_{L^2}\norm{h_{\delta}}_{L^2}\\
    &\lesssim\norm{L^{1/2}f}_{L^2}\norm{\nabla h_{\delta}}_{L^2}+\delta \norm{f}_{L^2}\norm{ h_{\delta}}_{L^2}\\
    &\lesssim\norm{L^{1/2}f}_{L^2}+\delta^{1/2} \norm{f}_{L^2},
\end{align*}
where we have used \eqref{Kato*} in the second inequality, and \eqref{hdelta L2gradient} in the last inequality.
Therefore,
\[
\norm{\nabla f}_{L^2}=\sup_{\substack{\bd g\in\mathscr{S}(\Rn)\\\norm{\bd g}_{L^2}\le 1}}\abs{\br{\nabla f,\bd g}}
\lesssim\norm{L^{1/2}f}_{L^2}+\delta^{1/2} \norm{f}_{L^2}.
\]
Letting $\delta\to0$, we obtain \eqref{Kato_reverse} holds for all $f\in \mathscr{S}(\Rn)$. Since $\mathscr{S}(\Rn)$ is dense in $W^{1,2}(\Rn)$, \eqref{Kato_reverse} holds for all $f\in W^{1,2}(\Rn)$, which contains the domain of $L$.

Finally, we show $D(L^{1/2})\subset W^{1,2}$, and thus proves $D(L^{1/2})= W^{1,2}$. To this end, let $u\in D(L^{1/2})$. Since $D(L)$ is a core of $L^{1/2}$, there exist $\set{u_n}\subset D(L)$ such that $u_n\to u$ in $L^2$, and $L^{1/2}(u_n)\to L^{1/2}u$ in $L^2$. Since
\[
\norm{\nabla (u_n-u_m)}_{L^2}\lesssim\norm{L^{1/2}(u_n-u_m)}_{L^2},
\]
$\set{u_n}$ is a Cauchy sequence in $W^{1,2}$. This implies $u\in W^{1,2}$.
 \end{proof}

From the Kato's estimate \eqref{Kato} one can see that $L^{1/2}$ can be extended to the homogeneous Sobolev space $\dot W^{1,2}$.
In particular, $L^{1/2}$ extends to an isomorphism from $\dot W^{1,2}$ to $L^2$ and
\begin{equation}\label{sqroot_isomorphism}
   g=L^{1/2}L^{-1/2}g \qquad\forall\, g\in L^2.
\end{equation}
In fact, by \eqref{Kato_reverse}, $L^{1/2}$ is one-to-one. So it suffices to justify that the range of $L^{1/2}$ is the whole $L^2$. To this end, we first show that the range of $L$ is dense in $L^2$.
\begin{lem}\label{RL_lem}
The range of $L$ is dense in $L^2$.
\end{lem}

\begin{proof}
Let $g\in L^2$. For any $\delta>0$, let $g_{\delta}\in\mathscr{S}(\Rn)$ such that $\norm{g_{\delta}-g}_{L^2}<\delta$. Define
\[
f^{(\delta)}_{\epsilon}:= \br{L+\epsilon I}^{-1}g_{\delta}\in D(L).
\]
We claim that $\norm{Lf^{(\delta)}_{\epsilon}-g}_{L^2}<C\delta$ when $\epsilon$ is sufficiently small. And this would complete the proof of the lemma.

We write
\[
Lf^{(\delta)}_{\epsilon}-g=L(L+\epsilon I)^{-1}g_{\delta}-g=g_{\delta}-\epsilon(L+\epsilon I)^{-1}g_{\delta}-g,
\]
and then
\begin{align}\label{RL_1}
    \norm{Lf^{(\delta)}_{\epsilon}-g}_{L^2}\le\norm{g_{\delta}-g}_{L^2}+\epsilon\norm{(L+\epsilon I)^{-1}g_{\delta}}<\delta+\epsilon\norm{(L+\epsilon I)^{-1}g_{\delta}}.
\end{align}
We have
\[
(L+\epsilon I)^{-1}g_{\delta}=\int_0^{\infty}e^{-t(L+\epsilon I)}(g_{\delta})dt,
\]
and thus
\[
\norm{(L+\epsilon I)^{-1}g_{\delta}}_{L^2}\le\int_0^{\infty}e^{-t\epsilon}\norm{e^{-tL}g_{\delta}}_{L^2}dt.
\]
Fix any $1<p<2$, then by the $L^p-L^2$ bound of the semigroup, we obtain
\begin{align*}
    \norm{(L+\epsilon I)^{-1}g_{\delta}}_{L^2}&\lesssim \int_0^{\infty}e^{-t\epsilon}t^{-\frac{\gamma_p}{2}}\norm{g_{\delta}}_{L^p}dt\\
    &\lesssim\epsilon^{\frac{\gamma_{p}}{2}-1}\int_0^{\infty}e^{-\tau}\tau^{-\frac{\gamma_p}{2}}d\tau\norm{g_{\delta}}_{L^p}\lesssim\epsilon^{\frac{\gamma_{p}}{2}-1}.
\end{align*}
By choosing $\epsilon$ sufficiently small, this and \eqref{RL_1} imply that $\norm{Lf^{(\delta)}_{\epsilon}-g}_{L^2}<2\delta$. Since $\delta>0$ is arbitrary, it proves that the range of $L$ is dense in $L^2$.
 \end{proof}

\begin{re}
By a similar argument and interpolation, one can show that the range of $L$ is dense in $L^p$, for any $1<p<\infty$.
\end{re}

\begin{cor}
The range of $L^{1/2}$ acting on $\dot{W}^{1,2}$ is $L^2$.
\end{cor}
\begin{proof}
Since $L^{1/2}L^{1/2}=L$ in $D(L)$, $L^{1/2}$ maps $D(L)$ into $D(L^{1/2})=W^{1,2}$, and $L^{1/2}(W^{1,2})$ contains the range of $L$. So the range of $L^{1/2}$ acting on $W^{1,2}$ is dense in $L^2$. Extending $L^{1/2}$ to $\dot{W}^{1,2}$, we claim that $L^{1/2}$ has closed range in $L^2$. To see this, suppose $\set{L^{1/2}f_n}$ is a Cauchy sequence in $L^2$ with $\lim_{n\to\infty}{L^{1/2}f_n}=y\in L^2$. By \eqref{Kato_reverse}, $\norm{\nabla(f_n-f_m)}_{L^2}\lesssim\norm{L^{1/2}(f_n-f_m)}_{L^2}$, which implies that $\set{f_n}$ is Cauchy in $\dot{W}^{1,2}$. So $f_n\to f\in \dot{W}^{1,2}$. Then we have
\begin{align*}
    \norm{y-L^{1/2}f}_{L^2}&\le\norm{y-L^{1/2}f_n}_{L^2}+\norm{L^{1/2}\br{f_n-f_m}}\\
    &<\epsilon+\norm{\nabla(f_n-f_m)}_{L^2}<2\epsilon
\end{align*}
for $n$ sufficiently large. This implies that $y=L^{1/2}f$, and thus the range of $L^{1/2}$ is closed.
 \end{proof}

A consequence of $D(L^{1/2})=\dot{W}^{1,2}(\Rn)$ is the following representation formula
\begin{lem}\label{L1/2repres Lem}
If $f$, $h\in\dot{W}^{1,2}$ then
\begin{equation*}
    \br{(L^*)^{1/2}f, L^{1/2}h} =\int_{\Rn}A\nabla f\cdot\overline{\nabla h}.
\end{equation*}
\end{lem}

\begin{proof}
For $f$, $h\in\dot{W}^{1,2}$, $L^{1/2}h$ and $(L^*)^{1/2}f$ belong to $L^2$. So both sides of the equality are well-defined (we use Proposition \ref{HardyProp1} for the right-hand side). Since the domain of $L$ is dense in $W^{1,2}$ and thus dense in $\dot{W}^{1,2}$, it suffices to show the equality holds for $h\in D(L)$ and $f\in\dot{W}^{1,2}$. By \eqref{sqrt_defid},
\[
\br{ (L^*)^{1/2}f, L^{1/2}h} =\br{ f, L^{1/2}L^{1/2}h} = \br{f, Lh},
\]
which equals to $\int_{\Rn}\nabla f\cdot\overline{A\nabla h}$ by construction of $L$.
\end{proof}

Another implication of \eqref{Kato_reverse} and \eqref{sqroot_isomorphism} is the $L^2$ boundedness of $\nabla L^{-1/2}$, the Riesz transform associated to $L$. In fact, since $L^{1/2}$ is an isomorphism from $\dot W^{1,2}$ to $L^2$, letting $f:= L^{-1/2}g\in \dot{W}^{1,2}$ in \eqref{Kato_reverse} one obtains
\begin{equation}\label{Riesztrans_L2}
    \norm{\nabla L^{-1/2}g}_{L^2}\lesssim\norm{g}_{L^2}\qquad\forall\, g\in L^2.
\end{equation}

Note that by the formula \eqref{squareroot_formula}, we immediately get the following formula for $L^{-1/2}$
\begin{equation*}
    L^{-1/2}g=\pi^{-1/2}\int_0^{\infty}e^{-tL}g\frac{dt}{\sqrt{t}}, \qquad\forall\, g\in R(L).
\end{equation*}

\begin{lem}\label{L-1/2formula lem nge3}
Let $n\ge 3$, and let $2^*=\frac{2n}{n-2}$. Let $1<p<\infty$ and $p\ne2$. Then for all $g\in L^2\cap L^p$,
\begin{equation}\label{L-1/2formula}
    L^{-1/2}g=\pi^{-1/2}\int_0^{\infty}e^{-tL}g\frac{dt}{\sqrt{t}}
\end{equation}
is valid and converges in $L^p+L^{2^*}$ if $p\ne 2^*$, and in $L^{2^*}+L^{p-\epsilon}$ if $p=2^*$, where $\epsilon>0$ is arbitrarily small.
\end{lem}

\begin{proof}
Let $g\in L^2\cap L^p$, and write
\[
\int_0^{\infty}e^{-tL}g\frac{dt}{\sqrt{t}}=\int_0^1e^{-tL}g\frac{dt}{\sqrt{t}}+\int_1^{\infty}e^{-tL}g\frac{dt}{\sqrt{t}}=: I+II.
\]
We first consider $1<p<2$. By the $L^p$ boundedness of the semigroup, $I$ converges in $L^p$ norm.
For $II$, note that we have
\begin{equation}\label{L-1/2 nge3 p<2}
    \norm{e^{-tL}g}_{L^{2^*}}\lesssim\norm{\nabla e^{-tL}g}_{L^2}\lesssim t^{-\frac{1+\gamma_p}{2}}\norm{g}_{L^p},
\end{equation}
which is a consequence of the $L^p-L^2$ bound of $(\sqrt{t}\nabla e^{-tL})_{t>0}$ and Sobolev embedding.
So the integral converges in $L^p+L^{2^*}$ norm.

For $p>2$, we consider the following cases.
\begin{enumerate}
    \item $p>2^*$. By the $L^{2^*}$ bound of $(e^{-tL})_{t>0}$, $\norm{I}_{L^{2^*}}\lesssim\norm{g}_{L^{2^*}}\le\norm{g}_{L^2\cap L^p}$, and thus $I$ converges in $L^{2^*}$. $II$ converges in $L^p$ because of the $L^2-L^p$ boundedness of $(e^{-tL})_{t>0}$. Note that $p>2^*$ gives $\gamma_p>1$, and thus $\int_1^{\infty}t^{-\frac{1+\gamma_p}{2}}dt<\infty$.
    \item $2<p<2^*$. For $I$, we use the $L^p-L^{2^*}$ boundedness of the semigroup to get
    \[
    \norm{I}_{L^{2^*}}\lesssim\int_0^1t^{-\frac{\gamma_{p2^*}+1}{2}}dt\norm{g}_{L^p}\lesssim\norm{g}_{L^p},
    \]
    where we have used $\gamma_{p2^*}=\frac{n}{p}-\frac{n-2}{n}<1$. That is, $I$ converges in $L^{2^*}$.
    Let $p_*=\frac{np}{n+p}$ be the reverse Sobolev exponent of $p$. Then by Sobolev embedding and $L^{p_*}-L^2$ boundedness of $(\sqrt{t}\nabla e^{-tL})_{t>0}$ (note that $p_*<2$), one gets
    \[
    \norm{e^{-tL}g}_{L^p}\lesssim\norm{\nabla e^{-tL}g}_{L^{p_*}}\lesssim t^{-\frac{1+\gamma_{p_*2}}{2}}\norm{g}_{L^2},
    \]
    which yields $II$ converges in $L^p$.
    \item $p=2^*$. One can see $I$ converges in $L^{2^*}$ from the $L^{2^*}$ boundedness of $(e^{-tL})_{t>0}$. For $II$, the Sobolev embedding and $L^{(p-\epsilon)_*}-L^2$ boundedness of $(\sqrt{t}\nabla e^{-tL})_{t>0}$ imply that $II$ converges in $L^{p-\epsilon}$, for arbitrary small $\epsilon>0$.
\end{enumerate}

Now it remains to show that the equality \eqref{L-1/2formula} holds for any $g\in L^2\cap L^p$.
By Lemma \ref{RL_lem} and the remark after it, $R(L)$ contains a dense subset of $L^2\cap L^p$. Therefore, there exists $\set{g_n}\subset R(L)\cap L^p$ such that $g_n\to g\in L^2\cap L^p$. Then \[L^{-1/2}g_n=\pi^{-1/2}\int_0^{\infty}e^{-tL}g_n\frac{dt}{\sqrt{t}}.\]
So from the convergence argument above, one can see that $\set{L^{-1/2}g_n}$ is a Cauchy sequence in $L^p+L^{2^*}$ if $p\ne2^*$. Suppose $L^{-1/2}g_n\to f\in L^p+L^{2^*}$. Since $L^{1/2}$ is an isomorphism from $\dot{W}^{1,2}$ to $L^2$, $L^{-1/2}g$ is well-defined. We compute
\begin{align*}
    \norm{L^{-1/2}g-f}_{L^p+L^{2^*}}&\le\norm{L^{-1/2}g-L^{-1/2}g_n}_{L^p+L^{2^*}}+\norm{L^{-1/2}g_n-f}_{L^p+L^{2^*}}\nonumber\\
    &\le \norm{L^{-1/2}g-L^{-1/2}g_n}_{L^{2^*}}+\norm{L^{-1/2}g_n-f}_{L^p+L^{2^*}}.
\end{align*}
And by Sobolev embedding and \eqref{Riesztrans_L2},
\[
 \norm{L^{-1/2}g-L^{-1/2}g_n}_{L^{2^*}}\lesssim\norm{\nabla L^{-1/2}(g_n-g)}_{L^2}\lesssim\norm{g_n-g}_{L^2}.
\]
Thus, we have proved that $L^{-1/2}g=f$. When $p=2^*$, $\set{L^{-1/2}g_n}$ is a Cauchy sequence in $L^{p-\epsilon}+L^{2^*}$, and a similar argument gives the same result.
 \end{proof}

\begin{lem}\label{L-1/2formula lem n=2}
Let $n=2$. Let $1<p<\infty$ and $p\ne2$. Then for all $g\in L^2\cap L^p$,
\begin{equation*}
    L^{-1/2}g=\pi^{-1/2}\int_0^{\infty}e^{-tL}g\frac{dt}{\sqrt{t}}
\end{equation*}
is valid and converges in $L^p(\Real^2)+BMO(\Real^2)$.
\end{lem}

\begin{proof}
As in the proof of Lemma \ref{L-1/2formula lem nge3}, we let $g\in L^2\cap L^p$ and write
\[
\int_0^{\infty}e^{-tL}g\frac{dt}{\sqrt{t}}=\int_0^1e^{-tL}g\frac{dt}{\sqrt{t}}+\int_1^{\infty}e^{-tL}g\frac{dt}{\sqrt{t}}=: I+II.
\]
When $1<p<2$, the same argument in the proof of Lemma \ref{L-1/2formula lem nge3} for $I$ carries over to the 2-d setting and shows that $I$ converges in $L^p(\Real^2)$. By replacing $L^{2^*}$ with $BMO$ in \eqref{L-1/2 nge3 p<2}, one gets $II$ converges in $BMO(\Real^2)$.

When $p>2$, the $L^p-L^{\infty}$ boundedness of the semigroup gives
\[
\norm{I}_{L^\infty}\lesssim\int_0^1t^{-\frac{1+\gamma_{p\infty}}{2}}dt\norm{g}_{L^p}\lesssim\norm{g}_{L^p},
\]
where we have used $\gamma_{p\infty}=\frac{2}{p}<1$. This implies that $I$ converges in $L^\infty$, and thus in $BMO$. $II$ converges in $L^p$ because of the Sobolev embedding and the $L^{p_*}-L^2$ boundedness of $(\sqrt{t}\nabla e^{-tL})_{t>0}$. Note that $p_*=\frac{2p}{2+p}<2$.

We have proved that for any $g\in L^2\cap L^p$, $\int_0^{\infty}e^{-tL}g\frac{dt}{\sqrt{t}}$ converges in $L^p+BMO$, and it remains to show the improper integral equals to $L^{-1/2}g$. To see this, one only needs to replace $L^{2^*}$ with $BMO$ in the corresponding proof of Lemma \ref{L-1/2formula lem nge3}.
 \end{proof}

\begin{cor}\label{Riesztransform_Rep cor}
We have the following representation for the Riesz transform:
\[
\br{\nabla L^{-1/2}f,\bd v}=\pi^{-1/2}\lim_{\epsilon\to 0}\int_{\Rn}\int_{\epsilon}^{\infty}\nabla e^{-tL}f  \,\frac{dt}{\sqrt{t}}\,\,\overline{\bd v}\,dx
\]
for all $f\in L^2\cap L^p$, and for all $\mathbb{C}^n$- valued $\bd v\in C_0^{\infty}$.

\end{cor}

\begin{proof}
Since we have observed in Lemma \ref{L-1/2formula lem nge3} and \ref{L-1/2formula lem n=2} that the improper integral defining $L^{-1/2}f$ converges in
$L^p + L^{2^*}$ or in $L^{p-\epsilon}+L^{2^*}$ if $p=2^*$, or in $L^p + BMO$ if $n=2$, and since $\divg {\bd v}$ is compactly supported and belongs to every
$L^p$, and to the Hardy space $\mathcal{H}^1$, we can write
\begin{align}
    \pi^{1/2}&\br{\nabla L^{-1/2}f,\bd v}=\pi^{1/2}\int_{\Rn}\nabla L^{-1/2}f\,\,\overline{\bd v}\,dx=-\pi^{1/2}\int_{\Rn}L^{-1/2}f\,\,\overline{\divg\bd v}\,dx\nonumber\\
    &=-\lim_{\epsilon\to 0}\int_{\Rn}\int_{\epsilon}^{\frac{1}{\epsilon}}e^{-tL}f  \,\frac{dt}{\sqrt{t}}\,\,\overline{\divg\bd v}\,dx
    =\lim_{\epsilon\to 0}\int_{\epsilon}^{\frac{1}{\epsilon}}\int_{\Rn}\nabla e^{-tL}f  \,\,\overline{\bd v}\,dx\frac{dt}{\sqrt{t}}\nonumber\\
    &=\lim_{\epsilon\to 0}\int_{\epsilon}^{1}\int_{\Rn}\nabla e^{-tL}f  \,\,\overline{\bd v}\,dx\frac{dt}{\sqrt{t}}+\int_1^{\infty}\int_{\Rn}\nabla e^{-tL}f  \,\,\overline{\bd v}\,dx\frac{dt}{\sqrt{t}}\label{inftyconverge}\\
    &=\lim_{\epsilon\to 0}\int_{\Rn}\int_{\epsilon}^1 \nabla e^{-tL}f  \,\frac{dt}{\sqrt{t}}\,\,\overline{\bd v}\,dx+\int_{\Rn}\int_1^{\infty} \nabla e^{-tL}f  \,\frac{dt}{\sqrt{t}}\,\,\overline{\bd v}\,dx\nonumber\\
    &=\lim_{\epsilon\to 0}\int_{\Rn}\int_{\epsilon}^{\infty} \nabla e^{-tL}f  \,\frac{dt}{\sqrt{t}}\,\,\overline{\bd v}\,dx\nonumber
\end{align}
where we have used the $L^p-L^2$ bound (or the $L^2-L^p$ bound if $p>2$) of $(\sqrt{t}\nabla e^{-tL})_{t>0}$, so the second improper integral converges, and we can then exchange the order of integration.
 \end{proof}

We now show that the limit can be taken inside the integral in Corollary \ref{Riesztransform_Rep cor} and thus we have the following formula for the Riesz transform associated to $L$:
\begin{prop}\label{Riesztransform prop}
Let $1<p<\infty$ and $p\ne 2$, then
\[
\nabla L^{-1/2}(f)=\frac{1}{\sqrt{\pi}}\int_0^\infty\frac{\nabla e^{-tL}f}{\sqrt{t}}dt,\qquad\forall\, f\in L^2\cap L^p.
\]
\end{prop}

\begin{proof}
By Corollary \ref{Riesztransform_Rep cor}, it suffices to show that for $f\in L^2\cap L^p$,
\begin{equation}\label{F_k Cauchy}
    F_k(f):= \int_{\frac{1}{k}}^\infty \frac{\nabla e^{-tL}f}{\sqrt{t}}dt \quad\text{is a Cauchy sequence in }L^2(\Rn).
\end{equation}
By the McIntosh-Yagi theorem (\cite{mcintosh1990operators} Theorem 1), we have
\[
\int_0^\infty \int_{\Rn}\abs{(tL)^{1/4}e^{-tL}f}^2\frac{dxdt}{t}\lesssim\norm{f}_{L^2(\Rn)}^2\qquad\forall\, f\in L^2(\Rn),
\]
which implies that for any $\epsilon>0$, there exists $\delta_0=\delta(f,\epsilon)>0$ such that
\begin{equation}\label{M-Y small}
    \int_0^{\delta_0}\int_{\Rn}\abs{(tL)^{1/4}e^{-tL}f}^2\frac{dxdt}{t}\le \epsilon^2.
\end{equation}
Set $N_0>1$ so that $\frac{1}{N_0}<\delta_0$, and choose $m>k>N_0$. We write
\[
F_m(f)-F_k(f)=\int_{\frac{1}{m}}^{\frac{1}{k}}\frac{\nabla e^{-tL}f}{\sqrt{t}}dt=\nabla L^{-1/2}\int_{\frac{1}{m}}^{\frac{1}{k}}\frac{L^{1/2}e^{-tL}f}{\sqrt{t}}dt.
\]
Since $\nabla L^{-1/2}$ is bounded on $L^2$,
\begin{align*}
    \norm{F_m(f)-F_k(f)}_{L^2}&\le C\norm{\int_{\frac{1}{m}}^{\frac{1}{k}}\frac{L^{1/2}e^{-tL}f}{\sqrt{t}}dt}_{L^2}\\
    &= C\sup_{g\in L^2, \norm{g}_{L^2}\le1}\abs{\br{\int_{\frac{1}{m}}^{\frac{1}{k}}\frac{L^{1/2}e^{-tL}f}{\sqrt{t}}dt,g}}.
\end{align*}
We compute
\begin{align*}
    \br{\int_{\frac{1}{m}}^{\frac{1}{k}}\frac{L^{1/2}e^{-tL}f}{\sqrt{t}}dt,g}&=\int_{\frac{1}{m}}^{\frac{1}{k}}\frac{1}{\sqrt{t}}\br{L^{1/2}e^{-tL/2}f,e^{-tL^*/2}g}dt\\
    &=\int_{\frac{1}{m}}^{\frac{1}{k}}\frac{1}{\sqrt{t}}\br{L^{1/4}e^{-tL/2}f,(L^*)^{1/4}e^{-tL^*/2}g}dt\\
    &=\int_{\frac{1}{m}}^{\frac{1}{k}}\frac{1}{t}\br{(tL)^{1/4}e^{-tL/2}f,(tL^*)^{1/4}e^{-tL^*/2}g}dt.
\end{align*}
Then by Cauchy-Schwartz, the McIntosh-Yagi theorem and \eqref{M-Y small}, we obtain
\begin{align*}
    &\quad\abs{\br{\int_{\frac{1}{m}}^{\frac{1}{k}}\frac{L^{1/2}e^{-tL}f}{\sqrt{t}}dt,g}}\\
    &\le\br{\int_{\frac{1}{m}}^{\frac{1}{k}}\int_{\Rn}\abs{(tL)^{1/4}e^{-tL/2}f}^2\frac{dxdt}{t}}^{1/2}
    \br{\int_{\frac{1}{m}}^{\frac{1}{k}}\int_{\Rn}\abs{(tL^*)^{1/4}e^{-tL^*/2}g}^2\frac{dxdt}{t}}^{1/2}\\
    &\le C\br{\int_0^{\delta_0}\int_{\Rn}\abs{(tL)^{1/4}e^{-tL/2}f}^2\frac{dxdt}{t}}^{1/2}\norm{g}_{L^2}\le C\epsilon.
\end{align*}
This implies that for any $m\ge k>N_0$, $\norm{F_m(f)-F_k(f)}_{L^2}\le C\epsilon$, that is, $\set{F_k(f)}$ is Cauchy in $L^2$.
 \end{proof}

We shall use this representation of the Riesz transform to obtain the $L^p$ boundedness of the square root operator.
\begin{lem}\label{lem to sqrtL* Lp}
To show that $\norm{(L^*)^{1/2}f}_{L^p}\lesssim\norm{\nabla f}_{L^p}$, $f\in \dot{W}^{1,p}$, with $p>1$, it suffices to show
\begin{equation}\label{Riesztransform Lp'}
    \norm{\nabla L^{-1/2}g}_{L^{p'}}\lesssim\norm{g}_{L^{p'}} \qquad\forall\, g\in L^2\cap L^{p'}.
\end{equation}
\end{lem}

\begin{proof}
Let $f\in\mathscr{S}(\Rn)$.
\begin{align*}
    \norm{(L^*)^{1/2}f}_{L^p}&=\sup_{\substack{g\in L^2\cap L^{p'}\\\norm{g}_{L^{p'}}\le1}}\abs{\br{(L^*)^{1/2}f,g}}
    =\sup_{\substack{g\in L^2\cap L^{p'}\\\norm{g}_{L^{p'}}\le1}}\abs{\br{(L^*)^{1/2}f,L^{1/2}L^{-1/2}g}}\\
    &=\sup_{\substack{g\in L^2\cap L^{p'}\\\norm{g}_{L^{p'}}\le1}}\abs{\br{(L^*)^{1/2}f,L^{1/2}h}}
\end{align*}
where $h=L^{-1/2}g$. Then by Lemma \ref{L1/2repres Lem} and Proposition \ref{HardyProp1}
\[
\abs{\br{(L^*)^{1/2}f,L^{1/2}h}}=\abs{\int_{\Rn}\nabla f\cdot\overline{A\nabla h}}\lesssim\norm{\nabla f}_{L^p}\norm{\nabla L^{-1/2}g}_{L^{p'}}.
\]
Therefore, \eqref{Riesztransform Lp'} gives
\begin{equation}\label{L*1/2Lp}
  \norm{(L^*)^{1/2}f}_{L^p}\lesssim\sup_{\substack{g\in L^2\cap L^{p'}\\\norm{g}_{L^{p'}}\le1}}\norm{\nabla f}_{L^p}\norm{g}_{L^{p'}}\lesssim\norm{\nabla f}_{L^p}\qquad\forall\, f\in\mathscr{S}(\Rn).
\end{equation}
Since $\mathscr{S}(\Rn)$ is dense in $\dot{W}^{1,p}$, $(L^*)^{1/2}$ can be extended to $\dot{W}^{1,p}$ and \eqref{L*1/2Lp} holds for all $f\in\dot{W}^{1,p}$. 
\end{proof}

Therefore, to prove $\norm{(L^*)^{1/2}f}_{L^p}\lesssim\norm{\nabla f}_{L^p}$ for $p>2$, it suffices to show
\begin{equation}\label{Riesztransform Lp' p_0}
    \norm{\nabla L^{-1/2}g}_{L^{p'}}\lesssim\norm{g}_{L^{p'}} \qquad\forall\, g\in L^2\cap L^{p_0},
\end{equation}
where $p_0\in [1,p')$. This is because $L^2\cap L^{p_0}$ is dense in $L^2\cap L^{p'}$. In order to prove \eqref{Riesztransform Lp' p_0}, we show that the Riesz transform is of weak type $(p_0,p_0)$, and then the strong type $(p,p)$ bound follows from interpolation with the strong type (2,2) bound \eqref{Riesztrans_L2}.

\begin{prop}\label{Riesztransform Lpbound Prop}
Let $p_0\in(1,2)$. Then
\begin{equation*}
    \norm{\nabla L^{-1/2}g}_{L^{\infty,p_0}}\lesssim\norm{g}_{L^{p_0}} \qquad\forall\, g\in L^{p_0}\cap L^2.
\end{equation*}
As a consequence, for any $1<p<2$,
\begin{equation*}
    \norm{\nabla L^{-1/2}g}_{L^p}\lesssim\norm{g}_{L^p} \qquad\forall\, g\in L^p\cap L^2.
\end{equation*}
\end{prop}

Proposition \ref{Riesztransform Lpbound Prop} can be proved exactly as in \cite{auscher2007necessary} (p.43-44) and its proof is thus omitted. The main ingredients in the proof are the Riesz transform representation formula (Proposition \ref{Riesztransform prop}), the $L^{p_0}-L^2$ off-diagonal estimates for $(\sqrt{t}\nabla e^{-tL})_{t>0}$, and the following lemma.

\begin{lem}[\cite{auscher2007necessary} Theorem 2.1]\label{AThem2.1}
Let $p_0\in[1,2)$. Suppose that $T$ is a sublinear operator of strong type $(2,2)$, and let $A_r$, $r>0$, be a family of linear operators acting on $L^2$. For a ball $B$, let $C_1(B)=4B$, $C_j(B)=2^{j+1}B\setminus 2^jB$ if $j\ge2$.

Assume for $j\ge2$
\begin{equation}\label{AThm2.1 2.1}
    \br{\frac{1}{\abs{2^{j+1}B}}\int_{C_j(B)}\abs{T(I-A_{r(B)})f}^2}^{1/2}\le g(j)\br{\frac{1}{\abs{B}}\int_B\abs{f}^{p_0}}^{1/{p_0}}
\end{equation}
and for $j\ge1$
\begin{equation}\label{AThm2.1 2.2}
  \br{\frac{1}{\abs{2^{j+1}B}}\int_{C_j(B)}\abs{A_{r(B)}f}^2}^{1/2}\le g(j)\br{\frac{1}{\abs{B}}\int_B\abs{f}^{p_0}}^{1/{p_0}}
\end{equation}
for all ball $B$ with radius $r(B)$ and all $f$ supported in $B$.
If $\Sigma=\Sigma_jg(j)2^{nj}<\infty$, then $T$ is of weak type $(p_0,p_0)$, with a bound depending only on the strong type $(2,2)$ bound of $T$, $p_0$ and $\Sigma$.
\end{lem}

We now turn to the case when $p<2$ in the $L^p$ estimate of the square root. Due to Lemma \ref{lem to sqrtL* Lp}, to show $\norm{(L^*)^{1/2}f}_{L^p}\lesssim\norm{\nabla f}_{L^p}$ for $1<p<2$, it suffices to show $\norm{\nabla L^{-1/2}g}_{L^{p'}}\lesssim\norm{g}_{L^{p'}}$ for any $g\in L^2\cap L^{p'}$.

\begin{prop}\label{Riesztransform Lpboundp>2 Prop}
Let $p_0\in(2,2+\epsilon_1)$, where $\epsilon_1$ is as in Proposition \ref{sqrttnabla_L2-LpProp}. Then for any $2<p<p_0$,
\begin{equation*}
    \norm{\nabla L^{-1/2}g}_{L^p}\lesssim\norm{g}_{L^p} \qquad\forall\, g\in L^p\cap L^2.
\end{equation*}
\end{prop}

The main ingredients of the proof are the Riesz transform representation formula we obtained in Proposition \ref{Riesztransform prop}, the $L^2-L^{p_0}$ off-diagonal estimate for $(\sqrt{t}\nabla e^{-tL})_{t>0}$, which gives the upper range $2+\epsilon_1$ of $p$,  and the following lemma. The proof of Proposition \ref{Riesztransform Lpboundp>2 Prop} is contained in \cite{auscher2007necessary} (p.48-50) and is thus omitted.

\begin{lem}[\cite{auscher2007necessary} Theorem 2.2]\label{AThm2.2}
Let $p_0\in (2,\infty]$. Suppose that $T$ is sublinear operator acting on $L^2$, and let $A_r$, $r>0$, be a family of linear operators acting on $L^2$. Assume
\begin{equation*}
    \br{\frac{1}{\abs{B}}\int_B\abs{T(I-A_{r(B)})f}^2}^{1/2}\le C\br{M(\abs{f}^2)}^{1/2}(y),
\end{equation*}
and
\begin{equation*}
    \br{\frac{1}{\abs{B}}\int_B\abs{TA_{r(B)}f}^{p_0}}^{1/{p_0}}\le C\br{M(\abs{Tf}^2)}^{1/2}(y),
\end{equation*}
for all $f\in L^2$, all ball $B$ with radius $r(B)$ and all $y\in B$. If $2<p<p_0$ and $Tf\in L^p$ when $f\in L^p\cap L^2$, then $T$ is strong type $(p,p)$. More precisely, for all $f\in L^p\cap L^2$,
\[
\norm{Tf}_{L^p}\le c\norm{f}_{L^p}
\]
where $c$ depends only on $n$, $p$, $p_0$ and $C$.
\end{lem}

We observe that the arguments above all applies to the adjoint operator $L^*$. So we have obtained
\begin{prop}
Let $1+\frac{1}{1+\epsilon_1}<p<\infty$, where $\epsilon_1$ is as in Proposition \ref{sqrttnabla_L2-LpProp}. Let $f\in \dot{W}^{1,p}$. Then
\begin{equation*}
    \norm{(L^*)^{1/2}f}_{L^p}\lesssim\norm{\nabla f}_{L^p}, \qquad \norm{L^{1/2}f}_{L^p}\lesssim\norm{\nabla f}_{L^p}.
\end{equation*}
\end{prop}

We remark that the $L^p$ estimate for the square root is actually valid for all $p\in(1,\infty)$. That is, we have
$\norm{L^{1/2}f}_{L^p}\lesssim\norm{\nabla f}_{L^p}$ for $1<p<\infty$, $f\in \dot{W}^{1,p}$.
This can be obtained by the weak type (1,1) estimate
\begin{equation}\label{L1/2 weak1}
    \norm{L^{1/2}f}_{L^{1,\infty}}\lesssim\norm{\nabla f}_{L^1}
\end{equation}
and then by Marcinkiewicz interpolation. \eqref{L1/2 weak1} can be derived using a Calderon-Zygmund decomposition for Sobolev functions. One can find the details in \cite{auscher2007necessary} Lemma 5.13. There, the result is carefully relied on dimension, as well as the lower range of $p$ in the $L^p-L^2$ off-diagonal estimate for the semigroup. In our setting, we do not need those discussions mainly because the Gaussian estimate \eqref{K_t_bd} yields the $L^p-L^2$ off-diagonal estimate for $(e^{-tL})_{t>0}$ for all $1<p<2$.

We end this section by summarizing our results. Note that by writing $\nabla f= \nabla L^{-1/2}\br{L^{1/2}f}$, the $L^p$ estimate for the Riesz transform associated to $L$ yields the invertibility property of the square root on $L^p$ spaces.
\begin{thm}\label{squareroot main thm}
Let $n\ge 2$, $\epsilon_1$ be as in Proposition \ref{sqrttnabla_L2-LpProp}, and let $f\in \dot{W}^{1,p}(\Rn)$. Then $\norm{L^{1/2}f}_{L^p}\lesssim\norm{\nabla f}_{L^p}$ for $1<p<\infty$. And $\norm{\nabla f}_{L^p}\lesssim\norm{L^{1/2}f}_{L^p}$ for $1<p<2+\epsilon_1$. Furthermore, $L^{1/2}$ extends to an isomorphism from $\dot{W}^{1,p}$ onto $L^p$ when $1<p<2+\epsilon_1$.
\end{thm}

\section{$L^p$ estimates for square functions}
\label{Lp est for square functions subsect}

\begin{prop}
For any $f\in L^2$,
\begin{equation}\label{PropgLL2}
    \int_{\Rn}\int_0^{\infty}\abs{(L^{1/2}e^{-tL}f)(x)}^2dtdx\lesssim \norm{f}_{L^2(\Rn)}^2,
\end{equation}
where the implicit constant depends on $\lambda_0$, $\Lambda_0$ and $n$. By a change of variable, we have
\begin{equation}\label{PropgLL2_t2}
   \int_{\Rn}\int_0^{\infty}\abs{tL^{1/2}e^{-t^2L}f(x)}^2\frac{dtdx}{t}\lesssim \norm{f}_{L^2(\Rn)}^2.
\end{equation}
\end{prop}

\begin{proof}
We first prove
\begin{equation}\label{GLL2}
    \int_0^{\infty}\int_{\Rn}\abs{(\nabla e^{-tL}f)(x)}^2dtdx\lesssim\norm{f}_{L^2(\Rn)}^2.
\end{equation}
Actually, the converse of \eqref{GLL2} is also true. We have
\begin{equation}\label{GLL2identity}
   \norm{f}_{L^2}^2=-\int_0^{\infty}\frac{d}{dt}\norm{e^{-tL}f}_{L^2}^2dt.
\end{equation}
We postpone its proof to the end.
Using Lemma \ref{Evanslem}, one has
\begin{align*}
    \Re \int_{\Rn}A\nabla e^{-tL}f\cdot\overline{\nabla e^{-tL}f}
    &=\Re\br{Le^{-tL}f, e^{-tL}f}=-\Re\act{\Dt e^{-tL}f,e^{-tL}f}\\
    &=-\frac{1}{2}\frac{d}{dt}\norm{e^{-tL}f}_{L^2}^2.
\end{align*}
Therefore,
\[
\norm{f}_{L^2}^2=2\int_0^{\infty}\Re\int_{\Rn}A\nabla e^{-tL}f\cdot\overline{\nabla e^{-tL}f}dxdt,
\]
and \eqref{GLL2} follows from ellipticity.

Using $\norm{L^{1/2}f}_{L^2}\lesssim\norm{\nabla f}_{L^2}$, \eqref{GLL2} gives
\begin{align*}
    \int_0^{\infty}\int_{\Rn}\abs{(L^{1/2}e^{-tL}f)(x)}^2dxdt
    \lesssim\int_0^{\infty}\int_{\Rn}\abs{\nabla e^{-tL}f}^2dxdt
    \lesssim \norm{f}_{L^2(\Rn)}^2
\end{align*}

\underline{Proof of \eqref{GLL2identity}}.
The only thing that needs clarification is
\begin{equation}\label{lim_t=0}
    \lim_{t\to\infty}\norm{e^{-tL}f}_{L^2}^2=0,\qquad\forall\, f\in L^2(\Rn).
\end{equation}

Fix any $f\in L^2(\Rn)$. For any $\eps>0$, choose $\vp_{\eps}\in C_0^{\infty}(\Rn)$ such that
\[
\norm{\vp_{\eps}}_{L^2}\le\norm{f}_{L^2}+1 \quad\text{  and  } \norm{e^{-tL}(f-\vp_{\eps})}_{L^2}<\eps, \quad\forall\,t>0.
\]
And suppose $\supp\vp_{\eps}\subset Q_{\eps}$. Let $p\in (1,2)$. Then the $L^p$ boundedness of $(e^{-tL})_{t>0}$ gives
\begin{equation*}
    \norm{e^{-tL}\vp_{\eps}}_{L^{p}}\le C\norm{\vp_{\eps}}_{L^p}\le C\abs{Q_{\eps}}^{\frac{2-p}{2p}}\norm{\vp_{\eps}}_{L^2}.
\end{equation*}
And the $L^2-L^{p'}$ boundedness of $(e^{-tL})_{t>0}$ gives
\begin{equation*}
    \norm{e^{-tL}\vp_{\eps}}_{L^{p'}}\le Ct^{-\frac{\gamma_{p'}}{2}}\norm{\vp_{\eps}}_{L^2}.
\end{equation*}
Therefore, we have
\begin{align*}
    \norm{e^{-tL}\vp_{\eps}}_{L^2}\le \norm{e^{-tL}\vp_{\eps}}_{L^p}^{1/2}\norm{e^{-tL}\vp_{\eps}}_{L^{p'}}^{1/2}
    \le C\abs{Q_{\eps}}^{\frac{2-p}{4p}}t^{-\frac{\gamma_{p'}}{4}}\br{\norm{f}_{L^2}+1}.
\end{align*}
Then there is a $t_0=t_0(\eps)>0$, such that for any $t>t_0$, $\norm{e^{-tL}\vp_{\eps}}_{L^2}<\eps$, and thus $\norm{e^{-tL}f}_{L^2}<2\eps$. This proves \eqref{lim_t=0}. 
\end{proof}

\begin{prop}\label{Lp_G1Prop}
\begin{equation}\label{Lp_G1}
\norm{\Big(\int_0^{\infty}\abs{tL e^{-t^2L}F}^2\frac{dt}{t}\Big)^{1/2}}_{L^p(\Rn)}\le C_p\norm{\nabla F}_{L^p(\Rn)}
\end{equation}
for all $1<p<\infty$, and $F\in W^{1,2}\cap W^{1,p}$. Equivalently,
\begin{equation}\label{Lp_G1Dt}
\norm{\Big(\int_0^{\infty}\abs{\Dt e^{-t^2L}F}^2\frac{dt}{t}\Big)^{1/2}}_{L^p(\Rn)}\le C_p\norm{\nabla F}_{L^p(\Rn)}.
\end{equation}
\end{prop}

\begin{proof}

We claim that to obtain \eqref{Lp_G1}, it suffices to show
\begin{equation}\label{g_L}
    \norm{\br{\int_0^\infty\abs{L^{1/2}e^{-tL}f}^2dt}^{1/2}}_{L^p}\le C_p\norm{f}_{L^p}\qquad\forall\, 1<p<\infty, \quad f\in L^2\cap L^p.
\end{equation}
In fact, by letting $f=L^{1/2}F$ in \eqref{g_L}, \eqref{Lp_G1} follows from the commutativity of $e^{-t^2L}$ and $L^{1/2}$, the $L^p$ estimate for the square root, and a change of variables. The square function $\br{\int_0^\infty\abs{L^{1/2}e^{-tL}f}^2dt}^{1/2}$ is defined to be $g_L(f)$ in \cite{auscher2007necessary}, and it has been proved to be $L^p$ bounded by $\norm{f}_{L^p}$ with the range of $p\in(1,\infty)$ same as the one of boundedness of the semigroup, up to endpoints. In other words, \eqref{g_L} holds for $1<p<\infty$. \eqref{g_L} can be proved (\cite{auscher2007necessary} p.78-80) using Lemma \ref{AThem2.1} for $p<2$ and Lemma \ref{AThm2.2} for $p>2$.
 \end{proof}

\begin{re}
For $2\le p<\infty$, \eqref{Lp_G1} can be alternatively proved by showing that $\abs{tLe^{-t^2L}F(x)}^2\frac{dxdt}{t}$ is a Carleson measure in $\Real^{n+1}_+$, and then using tent space interpolation, as well as local estimates for $\Dt e^{-t^2L}F(x)$. To show that $\abs{tLe^{-t^2L}F(x)}^2\frac{dxdt}{t}$ is a Carleson measure one needs the Gaussian decay estimate for the kernel of $\Dt e^{-t^2L}$ (see \eqref{Dtl_Kt_Gaussianbd}). This method is used in \cite{hofmann2015square} to show \eqref{Lp_G1} for $p\ge 2$ and for $L$ being an elliptic operator with $L^\infty$ coefficients.
\end{re}

We now derive $L^p$ estimates for the functional $\Big(\int_0^{\infty}\abs{t^2\nabla L e^{-t^2L}}^2\frac{dt}{t}\Big)^{1/2}$.

\begin{prop}\label{Lp_G2xProp}
Let $1<p<2+\epsilon_1$, where $\epsilon_1$ is as in Proposition \ref{sqrttnabla_L2-LpProp}. And let $F\in W^{1,2}\cap W^{1,p}$. Then
\begin{equation}\label{Lp_G2x}
\norm{\Big(\int_0^{\infty}\abs{t^2\nabla L e^{-t^2L}F}^2\frac{dt}{t}\Big)^{1/2}}_{L^p(\Rn)}\le C_p\norm{\nabla F}_{L^p(\Rn)}.
\end{equation} Or equivalently,
\begin{equation*}
\norm{\Big(\int_0^{\infty}\abs{t\nabla \Dt e^{-t^2L}F}^2\frac{dt}{t}\Big)^{1/2}}_{L^p(\Rn)}\le C_p\norm{\nabla F}_{L^p(\Rn)}.
\end{equation*}
\end{prop}

\begin{proof}
We shall establish the following
\begin{equation}\label{sqr}
  \norm{\Big(\int_0^{\infty}\abs{t^2\nabla L^{1/2}e^{-t^2L}f}^2\frac{dt}{t}\Big)^{1/2}}_{L^p(\Rn)}\le C_p\norm{f}_{L^p(\Rn)}
\end{equation}
for $1< p< 2+\epsilon_1$ and $f\in L^2\cap L^p$.
Once this is proved, setting $f=L^{1/2}F\in L^2\cap L^p$ and then using
\begin{equation*}
\norm{L^{1/2}F}_{L^p}\lesssim\norm{\nabla F}_{L^p}, \qquad\forall\, 1<p<\infty
\end{equation*}
one obtains \eqref{Lp_G2x}.\par
We now prove \eqref{sqr} by cases.\par
\textbf{Case 1: $p=2$.}
We write $e^{-t^2L}=e^{-t^2L/2}e^{-t^2L/2}$, and use the fact that $t\nabla e^{-t^2L/2}$ is bounded on $L^2(\Rn)$, uniformly in $t$ (by \eqref{nablaetL_L2bd}), to obtain
\begin{multline*}
    \int_{\Rn}\int_0^{\infty}\abs{t^2\nabla L^{1/2}e^{-t^2L}f}^2\frac{dt}{t}dx
   =\int_0^{\infty}t\int_{\Rn}\abs{t\nabla e^{-t^2L/2}(L^{1/2}e^{-t^2L/2}f)}^2dxdt\\
   \le C\int_0^{\infty}t\int_{\Rn}\abs{L^{1/2}e^{-t^2L/2}f}^2dxdt
   =C\int_{\Rn}\int_0^{\infty}t\abs{L^{1/2}e^{-t^2L/2}f}^2dtdx.
\end{multline*}
By Proposition \ref{PropgLL2},
\begin{equation*}
   \int_{\Rn}\int_0^{\infty}t\abs{L^{1/2}e^{-t^2L/2}f}^2dtdx\lesssim
   \norm{f}_{L^2(\Rn)}^2,
\end{equation*}
which finishes the proof of $L^2$ boundedness.

\textbf{Case 2: $2<p<2+\epsilon_1$.} We exploit Lemma \ref{AThm2.2} in this case. Note that Lemma \ref{AThm2.2} requires $Tf\in L^p$ when $f\in L^p\cap L^2$ and the purpose of the statement is to bound the $L^p$ norm of $Tf$. In practice, we would apply this lemma to suitable approximation of the operator $T$ and obtain uniform $L^p$ bounds. The uniformity of the bounds allows a limiting argument to deduce $L^p$ boundedness of $T$.

With this in mind, we define
$$
\mathcal{G}(f):= \Big(\int_0^{\infty}\abs{t^2\nabla L^{1/2}e^{-t^2L}f}^2\frac{dt}{t}\Big)^{1/2},
$$
and define $\mathcal{G}_{\eps}$ to be the approximation of $\mathcal{G}$
$$
\mathcal{G}_{\eps}(f):= \br{\int_{\eps}^{\infty}\abs{t^2\nabla L^{1/2}e^{-t^2L}f}^2\frac{dt}{t}}^{1/2}.
$$
We shall first show $\mathcal{G}_{\eps}(f)\in L^p$ for $f\in L^p\cap L^2$, and then derive the uniform estimates for $\mathcal{G}_{\eps}(f)$. Namely,
\begin{equation}\label{2.3}
\br{\frac{1}{\abs{B}}\int_B\abs{\mathcal{G}_{\eps}(I-e^{-r^2L})^mf}^2}^{1/2}\lesssim(M(\abs{f}^2))^{1/2}(y) \qquad\forall\, y\in B
\end{equation}
and
\begin{equation}\label{2.4}
\br{\frac{1}{\abs{B}}\int_B\abs{\mathcal{G}_{\eps}(e^{-r^2L}f)}^{p_0}}^{1/{p_0}}\lesssim(M(\abs{\mathcal{G}_{\eps}f}^2))^{1/2}(y) \qquad\forall\, y\in B
\end{equation}
uniformly in $\eps$, for any ball $B$ with radius $r$, and for some integer $m$ large enough. We shall prove \eqref{2.4} with $p_0=2+\epsilon_1$. Then by Lemma \ref{AThm2.2},
$\norm{\mathcal{G}_{\eps}(f)}_{L^p}\lesssim\norm{f}_{L^p}$ for all $2<p<p_0$, uniformly in $\eps$. Letting $\eps\to 0$, one obtains $\norm{\mathcal{G}f}_{L^p}\lesssim\norm{f}_{L^p}$ for all $2<p<p_0$.

\underline{Proof of $\mathcal{G}_{\eps}(f)\in L^p$ for $f\in L^p\cap L^2$.}

We rewrite $\mathcal{G}_{\eps}(f)$ to be
\[
\mathcal{G}_{\eps}(f)=\br{\int_{{\eps}^2}^{\infty}\abs{t\nabla L^{1/2}e^{-tL}f}^2\frac{dt}{t}}^{1/2}.
\]
By Minkowski's inequality,
\begin{align*}
    \norm{\mathcal{G}_{\eps}(f)}_{L^p}&\le\Big\{ \int_{\eps^2}^{\infty}\br{\int_{\Rn}\br{t\abs{\nabla L^{1/2}e^{-tL}f}^2}^{p/2}dx}^{2/p}dt\Big\}^{1/2}\\
    &=\Big\{ \int_{\eps^2}^{\infty}\br{\int_{\Rn}\abs{\sqrt{t}\nabla L^{1/2}e^{-tL}f}^pdx}^{2/p}dt\Big\}^{1/2}\\
    &=\Big\{ \int_{\eps^2}^{\infty}\br{\int_{\Rn}\abs{\sqrt{t}\nabla e^{-tL/2}L^{1/2}e^{-tL/2}f}^pdx}^{2/p}dt\Big\}^{1/2}.
\end{align*}
We first use the $L^2-L^p$ bounds for $(\sqrt{t}\nabla e^{-tL})_{t>0}$, then \eqref{Kato}, and finally the $L^2$ bounds for $(\sqrt{t}\nabla e^{-tL})_{t>0}$ to obtain
\begin{multline*}
    \norm{\mathcal{G}_{\eps}(f)}_{L^p}\lesssim \br{\int_{\eps^2}^{\infty}t^{-\gamma_p}\int_{\Rn}\abs{L^{1/2}e^{-tL/2}f}^2dxdt}^{1/2}\\
    \lesssim\br{\int_{\eps^2}^{\infty}t^{-\gamma_p}\int_{\Rn}\abs{\nabla e^{-tL/2}f}^2dxdt}^{1/2}\\
    \lesssim\br{\int_{\eps^2}^{\infty}t^{-\gamma_p-1}dt}^{1/2}\norm{f}_{L^2}
    \le C_{\eps, p}\norm{f}_{L^2}.
\end{multline*}

\underline{Proof of \eqref{2.4}}\par
Since the domain of $e^{-tL}$ is $L^2$, the operators commute for $f\in L^2$:
\begin{equation*}
 \nabla L^{1/2}e^{-t^2L}(e^{-r^2L}f)=\nabla e^{-r^2L}L^{1/2}e^{-t^2L}f.
\end{equation*}
By Minkowski inequality,
\begin{align}\label{G_2(2.4)Mink}
    \Big\{\frac{1}{\abs{B}}&\int_B\Big(\int_{\eps}^{\infty}\abs{t^2\nabla L^{1/2}e^{-t^2L}(e^{-r^2L}f)}^2\frac{dt}{t}\Big)^{\frac{p_0}{2}}dx\Big\}^{\frac{2}{p_0}}\nonumber\\
    &\le\int_{\eps}^{\infty}\Big(\frac{1}{\abs{B}}\int_B\abs{t^2\nabla L^{1/2}e^{-t^2L}(e^{-r^2L}f)}^{p_0}dx\Big)^{\frac{2}{p_0}}\frac{dt}{t}\nonumber\\
    &=\int_{\eps}^{\infty}\Big(\frac{1}{\abs{B}}\int_B\abs{t^2\nabla e^{-r^2L}(L^{1/2}e^{-t^2L}f)}^{p_0}dx\Big)^{\frac{2}{p_0}}\frac{dt}{t}
\end{align}
Using the $L^2-L^{p_0}$ off-diagonal estimates for $(\sqrt{r}\nabla e^{-rL})_{r>0}$, as well as Poincar\'e inequality, one can show
\begin{equation*}
    \Big(\frac{1}{\abs{B}}\int_B\abs{\nabla e^{-r^2L}f}^{p_0}\Big)^{1/{p_0}}\le\sum_{j\ge1}g(j)\Big(\frac{1}{\abs{2^{j+1}B}}
    \int_{2^{j+1}B}\abs{\nabla f}^2\Big)^{1/2}
\end{equation*}
with $\sum_{j\ge1}g(j)<\infty$. By this and H\"older inequality we can bound \eqref{G_2(2.4)Mink} by
\begin{align}
    \int_{\eps}^{\infty}&\Big\{\sum_{j\ge1}g(j)
    \Big(\frac{1}{\abs{2^{j+1}B}}\int_{2^{j+1}B}\abs{t^2\nabla L^{1/2}e^{-t^2L}f}^2dx\Big)^{1/2}\Big\}^2\frac{dt}{t}\nonumber\\
    &\le C\int_{\eps}^{\infty}\sum_{j\ge1}g(j)\frac{1}{\abs{2^{j+1}B}}\int_{2^{j+1}B}\abs{t^2\nabla L^{1/2}e^{-t^2L}f}^2dx\frac{dt}{t}\nonumber\\
    &\le C\sup_{j\ge1}\frac{1}{\abs{2^{j+1}B}}\int_{2^{j+1}B}\int_{\eps}^{\infty}\abs{t^2\nabla L^{1/2}e^{-t^2L}f}^2\frac{dt}{t}dx\nonumber\\
    &\le C(\mathcal{M}\abs{\mathcal{G}_{\eps}f}^2)^{1/2}(y)\qquad\forall\, y\in B.
\end{align}
\underline{Proof of \eqref{2.3}}\par
Now write $f=\sum_{j\ge1}f_j$, where
$$
\begin{cases}
f_j=(f-(f)_{4B})\mathbbm{1}_{2^{j+1}B\setminus 2^jB} \quad j\ge2,\\
f_1=(f-(f)_{4B})\mathbbm{1}_{4B}.
\end{cases}
$$
Then
$$
\br{\frac{1}{\abs{B}}\int_B\abs{\mathcal{G}_{\eps}(I-e^{-r^2L})^mf}^2}^{1/2}\le\sum_{j\ge1}\br{\frac{1}{\abs{B}}\int_B\abs{\mathcal{G}_{\eps}(I-e^{-r^2L})^mf_j}^2}^{1/2}.
$$
For $f_1$, the $L^2$ bound of $\mathcal{G}$ and that of $(I-e^{-r^2L})^m$ (the latter is a consequence of the holomorphic functional calculus on $L^2$) imply
\begin{multline*}
    \br{\frac{1}{\abs{B}}\int_B\abs{\mathcal{G}_{\eps}(I-e^{-r^2L})^mf_1}^2}^{1/2}\le C\br{\frac{1}{\abs{B}}\int_{\Rn}\abs{f_1}^2}^{1/2}\\
 \le C\br{\frac{1}{\abs{B}}\int_{4B}\abs{f}^2}^{1/2}\le C(\mathcal{M}\abs{f}^2)^{1/2}(y) \quad\forall\, y\in B.
\end{multline*}
For $f_j$ with $j\ge2$, let $\vp(z)=tz^{1/2}e^{-t^2z}(1-e^{-r^2z})^m$. Then (see e.g. \cite{auscher2007necessary} section 3.2)
\begin{equation*}
  tL^{1/2}e^{-t^2L}(1-e^{-r^2L})^m=\vp(L)
  =\int_{\Gamma_+}e^{-zL}\eta_+(z)dz+\int_{\Gamma_-}e^{-zL}\eta_-(z)dz,
\end{equation*}
where $\Gamma_{\pm}$ is the half-ray $\Real^+e^{\pm i(\frac{\pi}{2}-\theta)}$,
\begin{equation*}
    \eta_{\pm}(z)=\frac{1}{2\pi i}\int_{\gamma_{\pm}}e^{\zeta z}\vp(\zeta)d\zeta, \quad z\in\Gamma_{\pm},
\end{equation*}
with $\gamma_{\pm}$ being the half-ray $\Real^{+}e^{\pm i\nu}$, and
$0<\omega_0<\theta<\nu<\frac{\pi}{2}$, where $\omega_0$ is as in Proposition \eqref{L2off-diag Prop}.
One can show 
\begin{equation}\label{eta+_}
 \abs{\eta_{\pm}(z)}\le\frac{Ct}{(\abs{z}+t^2)^{3/2}}\inf(1,\frac{r^{2m}}{(\abs{z}+t^2)^m}), \quad z\in\Gamma_{\pm},
\end{equation}
whose proof is postponed to the end.
Then,
\begin{align}
 \frac{1}{\abs{B}}&\int_B\abs{\mathcal{G}_{\eps}(I-e^{-r^2L})^mf_j}^2dx\nonumber\\
 &=\int_{\eps}^{\infty}\frac{t}{\abs{B}}\int_B\abs{\int_{\Gamma_+}\nabla e^{-zL}\eta_+(z)dzf_j+\int_{\Gamma_-}\nabla e^{-zL}\eta_-(z)dzf_j}^2dxdt\nonumber\\
 &\le C\int_0^{\infty}\frac{t}{\abs{B}}\int_B\abs{\int_{\Gamma_+}\nabla e^{-zL}\eta_+(z)dzf_j}^2dxdt\nonumber\\ &\quad+\int_0^{\infty}\frac{t}{\abs{B}}\int_B\abs{\int_{\Gamma_-}\nabla e^{-zL}\eta_-(z)dzf_j}^2dxdt\nonumber\\
 &=: I_++I_-.
\end{align}
By Minkowski inequality and \eqref{eta+_},
\begin{align*}
    &I_+\le\int_0^{\infty}\frac{t}{\abs{B}}\Big\{\int_{\Gamma_+}\Big(\int_B\abs{\nabla e^{-zL}\eta_+(z)f_j}^2dx\Big)^{1/2}\abs{dz}\Big\}^2dt\\
    &\lesssim\int_0^{\infty}\frac{t}{\abs{B}}\Big\{\int_{\Gamma_+}\Big(\int_B\abs{\sqrt{z}\nabla e^{-zL}\frac{t}{\abs{z}^{1/2}(\abs{z}+t^2)^{3/2}}\frac{r^{2m}}{(\abs{z}+t^2)^m}f_j}^2dx\Big)^{1/2}\abs{dz}\Big\}^2dt.
\end{align*}
Since $z\in\Gamma_+=\Real^+e^{i(\frac{\pi}{2}-\theta)}$ and $\theta<\omega_0$, we can apply the $L^2-L^2$ off-diagonal estimates for $\br{\sqrt{z}\nabla e^{-zL}}_{z\in\Sigma_{\frac{\pi}{2}-\theta}}$, and bound the expression above by
\begin{equation}\label{G_2 5.4}
\int_0^{\infty}\frac{t}{\abs{B}}\Big(\int_{\Gamma_+}e^{-\frac{c4^jr^2}{\abs{z}}}\frac{t}{\abs{z}^{1/2}(\abs{z}+t^2)^{3/2}}\frac{r^{2m}}{(\abs{z}+t^2)^m}\abs{dz}\norm{f_j}_{L^2}\Big)^2dt
\end{equation}
We use the following lemma to estimate \eqref{G_2 5.4}
\begin{lem}[\cite{auscher2007necessary} Lemma 5.5]\label{Alem5.5}
Let $\gamma,\alpha\ge0$, $m>0$ be fixed parameters, and $c$ a positive constant. For some $C$ independent of $j\in\mathbb{N}$, $r,t>0$, the integral
$$
I=\int_0^{\infty}e^{-\frac{c4^jr^2}{s}}\frac{1}{s^{\gamma/2}}\frac{t^{\alpha}}{(s+t)^{1+\alpha}}\frac{r^{2m}}{(s+t)^m}ds
$$
satisfies the estimate
\begin{equation*}
    I\le \frac{C}{4^{jm}(2^jr)^{\gamma}}\inf\Big((\frac{t}{4^jr^2})^{\alpha},(\frac{4^jr^2}{t})^m\Big).
\end{equation*}
\end{lem}
Letting $\gamma=1$, $\alpha=\frac{1}{2}$, and $t$ replaced by $t^2$ in the lemma, we obtain
\begin{align*}
\eqref{G_2 5.4}&\lesssim\int_0^{\infty}\frac{t}{(4^{jm}(2^jr))^2}\inf\Big((\frac{t^2}{4^jr^2}),(\frac{4^jr^2}{t^2})^{2m}\Big)dt\frac{1}{\abs{B}}\int_{2^{j+1}B}\abs{f}^2dx\\
&\lesssim2^{jn}4^{-2jm}\int_0^{\infty}\frac{t}{(2^jr)^2}\inf\Big(\frac{t^2}{4^jr^2},(\frac{4^jr^2}{t^2})^{2m}\Big)dt\frac{1}{\abs{2^{j+1}B}}\int_{2^{j+1}B}\abs{f}^2dx\\
&\lesssim2^{j(n-4m)}\mathcal{M}\abs{f}^2(y)\quad\forall\, y\in B.
\end{align*}
$I_-$ can be estimated similarly. Choose $4m>n$, we get
$$
\frac{1}{\abs{B}}\int_B\abs{\mathcal{G}_{\eps}(I-e^{-r^2L})^mf}^2\lesssim\mathcal{M}\abs{f}^2(y)\quad\forall\, y\in B,
$$
which proves \eqref{2.3}.\\

\textbf{Case 3: $1<p<2$.} We use Lemma \ref{AThem2.1} to prove \eqref{Lp_G2x} holds for $1<p<2$. Define $\mathcal{G}$ as before. Then by letting $T=\mathcal{G}$ and $A_r=I-(I-e^{-r^2L})^m$ in Lemma \ref{AThem2.1}, it suffices to show the following: let $1<p_0<2$,
\begin{equation}\label{2.1}
\br{\frac{1}{\abs{2^{j+1}B}}\int_{C_j(B)}\abs{\mathcal{G}(I-e^{-r^2L})^mf}^2}^{1/2}\le g(j)\br{\frac{1}{\abs{B}}\int_B\abs{f}^{p_0}}^{1/{p_0}}\quad\text{for }j\ge2,
\end{equation}
and for $j\ge1$
\begin{equation}\label{2.2}
\br{\frac{1}{\abs{2^{j+1}B}}\int_{C_j(B)}\abs{(e^{-kr^2L}f)}^2}^{1/2}\le g(j)\br{\frac{1}{\abs{B}}\int_B\abs{f}^{p_0}}^{1/{p_0}}
\end{equation}
for any ball $B$ with radius $r$, for all $f$ supported in $B$, for some integer $m$ sufficiently large, $1\le k\le m$, and $\sum_jg(j)2^{nj}<\infty$.

\eqref{2.2} follows directly from the $L^{p_0}-L^2$ off-diagonal estimate of $(e^{-tL})_{t>0}$. We now turn to \eqref{2.1}.
As in the proof of \eqref{2.3}, we have
\begin{align}
\frac{1}{\abs{2^{j+1}(B)}}&\int_{C_j(B)}\abs{\mathcal{G}(I-e^{-r^2L})^mf}^2dx\nonumber\\
 &\le C\int_0^{\infty}\frac{t}{\abs{2^{j+1}B}}\int_{C_j(B)}\abs{\int_{\Gamma_+}\nabla e^{-zL}\eta_+(z)dz\,f}^2dxdt\nonumber\\ &\quad+\int_0^{\infty}\frac{t}{\abs{2^{j+1}B}}\int_{C_j(B)}\abs{\int_{\Gamma_-}\nabla e^{-zL}\eta_-(z)dz\,f}^2dxdt\nonumber\\
 &=: I_++I_-\label{G2p<2 I}
\end{align}
where $\eta_{\pm}$ and $\Gamma_{\pm}$ are same as before. We only estimate $I_+$, as $I_-$ can be estimated in a similar manner.
By Minkowski inequality and \eqref{eta+_}, $I_+$ is bounded up to a constant by
\begin{align*}
       \int_0^{\infty}\frac{t}{\abs{2^{j+1}B}}\Big\{\int_{\Gamma_+}\Big(\int_{C_j(B)}\abs{\sqrt{z}\nabla e^{-zL}\frac{t}{\abs{z}^{1/2}(\abs{z}+t^2)^{3/2}}\frac{r^{2m}}{(\abs{z}+t^2)^m}\,f}^2dx\Big)^{1/2}\abs{dz}\Big\}^2dt.
\end{align*}
The $L^{p_0}-L^2$ off-diagonal estimates for $\br{\sqrt{z}\nabla e^{-zL}}_{z\in\Sigma_{\frac{\pi}{2}-\theta}}$ implies that the expression above is bounded up to a constant by
\begin{equation*}
\int_0^{\infty}\frac{t}{\abs{2^{j+1}B}}\Big(\int_{\Gamma_+}e^{-\frac{c4^jr^2}{\abs{z}}}\frac{t}{\abs{z}^{(1+\gamma_{p_0})/2}(\abs{z}+t^2)^{3/2}}\frac{r^{2m}}{(\abs{z}+t^2)^m}\abs{dz}\norm{f}_{L^{p_0}(B)}\Big)^2dt.
\end{equation*}
Applying Lemma \ref{Alem5.5} with $\gamma=1+\gamma_{p_0}$ and $\alpha=\frac{1}{2}$, this is bounded up to a constant by
\begin{align*}
    4^{-2jm}&(2^jr)^{-2\gamma_{p_0}}\frac{\norm{f}^2_{L^{p_0}(B)}}{\abs{2^{j+1}B}}\int_0^\infty\frac{t}{(2^jr)^2}\inf\br{\frac{t^2}{4^jr^2},\br{\frac{4^jr^2}{t^2}}^{2m}}dt\\
    &\lesssim 2^{-j(4m+n+2\gamma_{p_0})}\br{\frac{1}{\abs{B}}\int_B\abs{f}^{p_0}}^{\frac{2}{p_0}}.
\end{align*}
Combining this with \eqref{G2p<2 I} gives \eqref{2.1} with $g(j)=2^{-j(2m+\frac{n}{2}+\gamma_{p_0})}$. And thus by choosing $m$ to be an integer such that $2m+\gamma_{p_0}>\frac{n}{2}$, we obtain the desired result.

\underline{Proof of \eqref{eta+_}}\par
We only show the estimate for $\eta_+$. The proof for $\eta_-$ is similar. \par
Write $\zeta=\rho e^{i\nu}$, and $z=\abs{z}e^{i(\frac{\pi}{2}-\theta)}$.
Then $\abs{e^{\zeta z}}=e^{-\rho\abs{z}\sin{(\nu-\theta)}}$, $\abs{e^{-t^2\zeta}}=e^{-t^2\rho\cos{\nu}}$. Since $0<\theta<\nu<\frac{\pi}{2}$,
$\abs{e^{\zeta z}e^{-t^2\zeta}}\le e^{-c\rho(\abs{z}+t^2)}$ for some $0<c<1$. So
$$
\abs{\eta_+(z)}\lesssim t\int_0^{\infty} \rho^{1/2}e^{-c\rho(\abs{z}+t^2)}H(\rho)^md\rho,
$$
where $H(\rho)=\abs{1-\exp(-r^2\rho e^{i\nu})}$. Observe that $H(0)=0$, $H(\rho)\le 2$, and that $H$ is a Lipschitz function with $[H]_{C^{0,1}}\le r^2$. So we have
$$
H(\rho)\le C\inf(1,r^2\rho).
$$
Using this estimate of $H$, we can bound $\abs{\eta_+(z)}$ by
\begin{equation*}
    Ct\int_0^{\infty}\rho^{\frac{1}{2}}e^{-c\rho(\abs{z}+t^2)}d\rho
    =\frac{Ct}{(\abs{z}+t^2)^{3/2}}\int_0^{\infty}s^{\frac{1}{2}}e^{-s}ds=\frac{Ct\Gamma(\frac{3}{2})}{(\abs{z}+t^2)^{3/2}},
\end{equation*}
and by
\begin{equation*}
    Ctr^{2m}\int_0^{\infty}\rho^{\frac{1}{2}+m}e^{-c\rho(\abs{z}+t^2)}d\rho
    =\frac{Ct\Gamma(m+\frac{3}{2})r^{2m}}{(\abs{z}+t^2)^{m+3/2}}.
\end{equation*}
Combining the two bounds we obtain
$$
\abs{\eta_+(z)}\le \frac{Ct}{(\abs{z}+t^2)^{3/2}}\inf(1,\frac{r^{2m}}{(\abs{z}+t^2)^m}).
$$ 
\end{proof}

We also have a similar estimate when the derivative falls on $t$. But in this case, the $L^p$ estimates hold for any $1<p<\infty$:
\begin{prop}\label{Lp_G2t prop}
\begin{equation}\label{Lp_G2t}
\norm{\Big(\int_0^{\infty}\abs{t^2\Dt L e^{-t^2L}F}^2\frac{dt}{t}\Big)^{1/2}}_{L^p(\Rn)}\le C_p\norm{\nabla F}_{L^p(\Rn)}
\end{equation}
for all $1<p<\infty$, and all $F\in W^{1,2}\cap W^{1,p}$.
\end{prop}

\begin{proof}
Let
\begin{equation}\label{Gdef G2t}
   \mathcal{G}(f)=\Big(\int_0^{\infty}\abs{t^2\Dt L^{1/2}e^{-t^2L}f}^2\frac{dt}{t}\Big)^{1/2}.
\end{equation}
Then by the $L^p$ bounds of the square root of $L$, it suffices to show
\begin{equation}\label{G2t_Gt_Lp}
\norm{\mathcal{G}f}_{L^p}\le C_p\norm{f}_{L^p}.
\end{equation}
\textbf{Case 1: $p=2$.} The argument can be copied almost verbatim from the proof in Proposition \ref{Lp_G2xProp}. The only difference is that we would use the uniform $L^2$ boundedness of $(t\Dt e^{-t^2L/2})_{t>0}$, rather than that of $(t\nabla e^{-t^2L/2})_{t>0}$.\\
\textbf{Case 2: $p>2$.} We shall apply Lemma \ref{AThm2.2} again. And as in the proof of Proposition \ref{Lp_G2xProp}, we should derive the analog of \eqref{2.3} and \eqref{2.4} for the approximation operator $\mathcal{G}_{\eps}$. We omit this limiting process here for simplicity, and only derive the analogous estimates for $\mathcal{G}$.

Let $p_0>2$. We wish to prove
\begin{equation*}
\br{\frac{1}{\abs{B}}\int_B\abs{\mathcal{G}(e^{-r^2L}f)}^{p_0}}^{1/{p_0}}\lesssim(M(\abs{\mathcal{G}f}^2))^{1/2}(y) \qquad\forall\, y\in B.
\end{equation*}
To this end, we first claim that
\begin{equation}\label{commute}
  t^2\Dt L^{1/2}e^{-t^2L}(e^{-r^2L}f)=-2t^3e^{-r^2L}L^{1/2}Le^{-t^2L}f, \qquad\forall\, f\in L^2.
\end{equation}
Then by this and Minkowski inequality,
\begin{align*}
  \Big\{\frac{1}{\abs{B}}&\int_B\Big(\int_0^{\infty}\abs{t^2\Dt L^{1/2}e^{-t^2L}(e^{-r^2L}f)}^2\frac{dt}{t}\Big)^{\frac{p_0}{2}}dx\Big\}^{\frac{2}{p_0}}\\
    &\le\int_0^{\infty}\Big(\frac{1}{\abs{B}}\int_B\abs{t^2\Dt L^{1/2}e^{-t^2L}(e^{-r^2L}f)}^{p_0}dx\Big)^{\frac{2}{p_0}}\frac{dt}{t}\\
    &=4\int_0^{\infty}\Big(\frac{1}{\abs{B}}\int_B\abs{t^3 e^{-r^2L}L^{1/2}Le^{-t^2L}f}^{p_0}dx\Big)^{\frac{2}{p_0}}\frac{dt}{t}.
\end{align*}
Using the $L^2-L^{p_0}$ off-diagonal estimates for $(e^{-r^2L})_{r>0}$, one can show
\begin{equation*}
    \Big(\frac{1}{\abs{B}}\int_B\abs{e^{-r^2L}f}^{p_0}\Big)^{2/{p_0}}\le\sum_{j\ge1}\frac{C_j}{\abs{2^{j+1}B}}\int_{2^{j+1}B}\abs{f}^2
\end{equation*}
with $C_j=Ce^{-c4^j}$. So
\begin{align*}
  \Big\{\frac{1}{\abs{B}}&\int_B\Big(\int_0^{\infty}\abs{t^2\Dt L^{1/2}e^{-t^2L}(e^{-r^2L}f)}^2\frac{dt}{t}\Big)^{\frac{p_0}{2}}dx\Big\}^{\frac{2}{p_0}}\\
  &\le 4\int_0^{\infty}\sum_{j\ge1}\frac{C_j}{\abs{2^{j+1}B}}\int_{2^{j+1}B}\abs{t^3L^{1/2}Le^{-t^2L}f}^2dx\frac{dt}{t}\\
  &\lesssim\sup_j\frac{1}{\abs{2^{j+1}B}}\int_{2^{j+1}B}\int_0^{\infty}\abs{t^2\Dt L^{1/2}e^{-t^2L}f}^2\frac{dt}{t}dx\\
  &\lesssim M(\abs{\mathcal{G}(f)}^2)(y)\qquad\forall\, y\in B.
\end{align*}
Now we prove
\begin{equation}\label{G2t 2.1 p>2}
\br{\frac{1}{\abs{B}}\int_B\abs{\mathcal{G}(I-e^{-r^2L})^mf}^2}^{1/2}\lesssim(M(\abs{f}^2))^{1/2}(y) \qquad\forall\, y\in B.
\end{equation}
Define $\{f_j\}_{j\ge1}$ as in the proof of Proposition \ref{Lp_G2xProp}. The estimates for $f_1$ again follows from the $L^2$ bound of $\mathcal{G}$ and that of $(I-e^{-r^2L})^m$. For $f_j$ with $j\ge2$, we let $\vp(z)=t^3z^{3/2}e^{-t^2z}(1-e^{-r^2z})^m$. Then
\begin{align*}
    -\frac{1}{2}t^2&\Dt L^{1/2}e^{-t^2L}(I-e^{-r^2L})^m=t^3L^{1/2}Le^{-t^2L}(I-e^{-r^2L})^m\\
    &=\vp(L)=\int_{\Gamma_+}e^{-zL}\eta_+(z)dz+\int_{\Gamma_-}e^{-zL}\eta_-(z)dz,
\end{align*}
where $\eta_{\pm}$ and $\Gamma_{\pm}$ are defined as in the proof of Proposition \ref{Lp_G2xProp}. By a similar argument as in the proof of \eqref{eta+_}, one can show
\begin{equation}\label{eta_t}
 \abs{\eta_{\pm}(z)}\lesssim t^3\int_0^{\infty}\rho^{3/2}e^{-c\rho(\abs{z}+t^2)}H(\rho)^md\rho
  \le \frac{Ct^3}{(\abs{z}+t^2)^{5/2}}\inf(1,\frac{r^{2m}}{(\abs{z}+t^2)^m}).
\end{equation}
Then,
\begin{multline*}
    \frac{1}{\abs{B}}\int_B\abs{\mathcal{G}(I-e^{-r^2L})^mf_j}^2\\
 =4\int_0^{\infty}\frac{1}{\abs{B}}\int_B\abs{t^3L^{1/2}Le^{-t^2L}(I-e^{-r^2L})^mf_j}^2dx\frac{dt}{t}\\
 =4\int_0^{\infty}\frac{1}{\abs{B}}\int_B\abs{\int_{\Gamma_+}e^{-zL}\eta_+(z)dzf_j
 +\int_{\Gamma_-}e^{-zL}\eta_-(z)dzf_j}^2dx\frac{dt}{t}\\
 \lesssim\int_0^{\infty}\frac{1}{\abs{B}}\int_B\abs{\int_{\Gamma_+} e^{-zL}\eta_+(z)dzf_j}^2dx\frac{dt}{t} \\+\int_0^{\infty}\frac{1}{\abs{B}}\int_B\abs{\int_{\Gamma_-} e^{-zL}\eta_-(z)dzf_j}^2dx\frac{dt}{t}
 =: I_++I_-.
\end{multline*}
By Minkowski inequality and \eqref{eta_t},
\begin{equation*}
  I_+\lesssim\int_0^{\infty}
  \frac{1}{\abs{B}}\Big\{\int_{\Gamma_+}\Big(\int_B\abs{t^3e^{-zL}\frac{r^{2m}}{(\abs{z}+t^2)^{\frac{5}{2}+m}}f_j}^2dx\Big)^{1/2}\abs{dz}\Big\}^2\frac{dt}{t}.
\end{equation*}
We use the $L^2$ off-diagonal estimate for $(e^{zL})_{z\in\Sigma_{\frac{\pi}{2}-\theta}}$ to bound the above expression by
$$
\int_0^{\infty}\frac{1}{\abs{B}}
\Big(\int_{\Gamma_+}e^{-\frac{c4^jr^2}{\abs{z}}}\frac{t^3r^{2m}}{(\abs{z}+t^2)^{5/2+m}}\abs{dz}\norm{f_j}_{L^2}\Big)^2\frac{dt}{t}.
$$

Applying Lemma \ref{Alem5.5} and letting $\alpha=\frac{3}{2}$, $\gamma=0$ and $t=t^2$ there, we obtain
\begin{align*}
    I_+&\lesssim\int_0^{\infty}(\frac{1}{4^{jm}})^2\inf\Big((\frac{t^2}{4^jr^2})^3,(\frac{4^jr^2}{t^2})^{2m}\Big)
    \frac{1}{\abs{B}}\norm{f}^2_{L^2(2^{j+1}B)}\frac{dt}{t}\\
    &\lesssim \frac{2^{jn-4jm}}{\abs{2^{j+1}B}}\norm{f}_{L^2(2^{j+1}B)}^2
    \Big(\int_0^{2^jr}\frac{t^5}{(4^jr^2)^3}dt+\int_{2^jr}^{\infty}\frac{(4^jr^2)^{2m}}{t^{4m+1}}dt\Big)\\
    &\lesssim 2^{jn-4jm}M(\abs{f}^2)(y) \qquad\forall\, y\in B.
\end{align*}
$I_-$ can be estimated similarly. Choose $4m>n$, we get
$$
\frac{1}{\abs{B}}\int_B\abs{\mathcal{G}(I-e^{-r^2L})^mf}^2\lesssim\mathcal{M}\abs{f}^2(y)\quad\forall\, y\in B.
$$
Therefore, \eqref{G2t_Gt_Lp} holds for all $2\le p<p_0$. And since $p_0>2$ is arbitrary, \eqref{G2t_Gt_Lp} holds for all $2\le p<\infty$.\\
\textbf{Case 3: $1<p<2$.} Applying Lemma \ref{AThem2.1} and letting $1<p_0<2$, it suffices to show \eqref{2.1} and \eqref{2.2}, where $\mathcal{G}$ is defined in \eqref{Gdef G2t}. Note that \eqref{2.2} is independent of the operator $\mathcal{G}$ and is verified in the proof of Proposition \ref{Lp_G2xProp}. To see \eqref{2.1}, we proceed similarly as the proof of \eqref{G2t 2.1 p>2}. Using the $L^{p_0}-L^2$ off-diagonal estimate of $(e^{zL})_{z\in\Sigma_{\frac{\pi}{2}-\theta}}$, we obtain
\begin{align*}
    \frac{1}{\abs{2^{j+1}B}}&\int_{C_j(B)}\abs{\mathcal{G}(I-e^{-r^2L})^mf}^2\\
    &\lesssim\frac{1}{\abs{2^{j+1}B}}\int_0^\infty\Big\{\int_{\Gamma_+}\abs{z}^{-\frac{\gamma_{p_0}}{2}}e^{-\frac{c(2^jr)^2}{\abs{z}}}
    \frac{t^3r^{2m}}{(\abs{z}+t^2)^{\frac{5}{2}+m}}\norm{f}_{L^{p_0}(B)}\abs{dz}\Big\}^2\frac{dt}{t}
\end{align*}
plus an integral over $\Gamma_-$ with the same integrand.
Applying Lemma \ref{Alem5.5} with $\gamma=\gamma_{p_0}$ and $\alpha=\frac{3}{2}$, we have
\begin{align*}
   \frac{1}{\abs{2^{j+1}B}}&\int_{C_j(B)}\abs{\mathcal{G}(I-e^{-r^2L})^mf}^2\\
   &\lesssim 4^{-2jm}(2^jr)^{-2\gamma_{p_0}}\frac{\norm{f}_{L^{p_0}(B)}^2}{\abs{2^{j+1}B}}\int_0^\infty\inf\br{\br{\frac{t^2}{4^jr^2}}^3,\br{\frac{4^jr^2}{t^2}}^{2m}}\frac{dt}{t}\\
   &\lesssim 2^{-j(n+4m+2\gamma_{p_0})}r^{-2\gamma_{p_0}}\frac{\norm{f}^2_{L^{p_0}(B)}}{\abs{B}}\\
   &\lesssim 2^{-j(n+4m+2\gamma_{p_0})}\br{\frac{1}{\abs{B}}\int_B\abs{f}^{p_0}}^{\frac{2}{p_0}}.
\end{align*}
Choosing $m$ to be an integer such that $2m+\gamma_{p_0}>\frac{n}{2}$ gives \eqref{2.1}. And this shows that \eqref{G2t_Gt_Lp} holds for all $1<p<2$.

\underline{Proof of \eqref{commute}.}\par
We know that for any $g\in L^2$, $e^{-t^2L}g\in D(L)$, and $\Dt e^{-t^2L}g\in W^{1,2}=D(L^{1/2})$. The latter follows from analyticity of the semigroup
$$\partial_je^{-t^2L}g=\frac{1}{2\pi i}\int_{\Gamma}e^{\lambda t^2}\partial_j(\lambda I+L)^{-1}(g)d\lambda,$$
as taking the derivative in $t$ gives that
$$
\Dt\partial_je^{-t^2L}g=\frac{2t}{2\pi i}\int_{\Gamma}e^{\lambda t^2}\lambda\partial_j(\lambda I+L)^{-1}(g)d\lambda\in L^2.
$$
Since $e^{-r^2L}f\in D(L)\subset D(L^{1/2})$ and that $e^{-t^2L}$ is a bounded, linear operator on $L^2$, the lemma implies
\begin{equation}\label{commute1}
  t^2\Dt L^{1/2}e^{-t^2L}(e^{-r^2L}f)=t^2 \Dt e^{-t^2L}L^{1/2}(e^{-r^2L}f).
\end{equation}
So
\begin{align}
   t^2\Dt &L^{1/2}e^{-t^2L}(e^{-r^2L}f)=-2t^3Le^{-t^2L}(L^{1/2}e^{-r^2L}f)\nonumber\\
  &=-2t^3L^{1/2}L^{1/2}e^{-t^2L}(L^{1/2}e^{-r^2L}f)=-2t^3L^{1/2}Le^{-t^2L}e^{-r^2L}f\label{commute2}
\end{align}
where the last equality follows from Lemma \ref{lem Claim8} and $L^{1/2}L^{1/2}=L$.
\begin{align}
    -2t^3&L^{1/2}Le^{-t^2L}e^{-r^2L}f=-2t^3L^{1/2}Le^{-r^2L}e^{-t^2L}f\nonumber\\
    &=-2t^3L^{1/2}e^{-r^2L}Le^{-t^2L}f=t^2L^{1/2}e^{-r^2L}\Dt e^{-t^2L}f\nonumber\\
    &=-2t^3e^{-r^2L}L^{1/2}Le^{-t^2L}f\label{commute3},
\end{align}
where in the last step we have used $-2tLe^{-t^2L}f=\Dt e^{-t^2L}f\in W^{1,2}=D(L^{1/2})$ and Lemma \ref{lem Claim8}. Combining \eqref{commute1}-\eqref{commute3}, we have proved \eqref{commute}. 
\end{proof}

\appendix
\section{Appendix}
We include some frequently used results in this appendix for reader's convenience.
\begin{lem}[\cite{giaquinta1983multiple} Chapter V Proposition 1.1]\label{GiaProp1.1}
Let $Q$ be a cube in $\Rn$. Let $g\in L^q(Q)$, $q>1$, and $f\in L^s(Q)$, $s>q$, be two nonnegative functions. Suppose
\[
\fint_{Q_R(x_0)}g^qdx\le b\br{\fint_{Q_{2R}(x_0)}gdx}^q+\fint_{Q_{2R}(x_0)}f^qdx+\theta\fint_{Q_{2R}(x_0)}g^qdx
\]
for each $x_0\in Q$ and each $R<\min\set{\frac{1}{2}\dist (x_0,\bdy Q),R_0}$, where $R_0$, $b$, $\theta$ are constants with $b>1$, $R_0>0$, $0\le\theta<1$. Then $g\in L_{\loc}^p(Q)$ for $p\in[q,q+\epsilon)$ and
\[
\br{\fint_{Q_R}g^pdx}^{1/p}\le c\Big\{\br{\fint_{Q_{2R}}g^qdx}^{1/q}+\br{\fint_{Q_{2R}}f^pdx}^{1/p}\Big\}
\]
for $Q_{2R}\subset Q$, $R<R_0$, where $c$ and $\epsilon$ are positive constants depending only on $b$, $\theta$, $q$, $n$ (and $s$).

\end{lem}

\begin{lem}\label{Evanslem}
Suppose $u,v\in L^2\br{(0,T),W^{1,2}(\Rn)}$ with $\Dt u, \Dt v \in L^2\br{(0,T),\wt{W}^{-1,2}(\Rn)}$. Then
\begin{enumerate}[(i)]
    \item $u\in C\br{[0,T], L^2(\Rn)}$;
    \item The mapping $t\mapsto\norm{u(\cdot,t)}_{L^2(\Rn)}$ is absolutely continuous, with
    \[\frac{d}{dt}\norm{u(\cdot,t)}_{L^2(\Rn)}^2=2\Re\act{\Dt u(\cdot,t),u(\cdot,t)} \quad\text{for a.e. }t\in[0,T].\]
    As a consequence,
    \[
    \frac{d}{dt}\br{u(\cdot,t),v(\cdot,t)}_{L^2(\Rn)}=\act{\Dt u(\cdot,t),v(\cdot,t)}+\overline{\langle{\Dt v(\cdot,t),u(\cdot,t)\rangle}}_{\wt{W}^{-1,2},W^{1,2}}\quad\text{a.e.}.
    \]
\end{enumerate}
\end{lem}
For its proof see e.g. \cite{evans1998partial} Section 5.9.2 Theorem 3.

\begin{lem}[\cite{kato1976perturbation} Chapter V, Theorem 3.35]\label{lem Claim8}
 For any bounded linear operator $B$ on $L^2$, if $BL=LB$ in $D(L)$, then $L^{1/2}B=BL^{1/2}$ in $D(L^{1/2})$.
 \end{lem}

\bibliographystyle{plain}
\bibliography{references}

\Addresses

\end{document}